\documentclass[twoside,12pt]{article}
\usepackage{graphicx,amsmath,latexsym,amssymb,amsthm}
 \usepackage{epstopdf}
\usepackage[colorlinks=black,urlcolor=black]{hyperref}
\font\chuto=cmbx10 at 16pt \font\kamy=lcmssb8
at 9pt

\newtheorem{theorem}{Theorem}[section]
\newtheorem{proposition}[theorem]{Proposition}
\newtheorem{definition}[theorem]{Definition}
\newtheorem{corollary}[theorem]{Corollary}
\newtheorem{lemma}[theorem]{Lemma}
\newtheorem{example}[theorem]{Example}
\newtheorem{remark}[theorem]{Remark}

\newtheorem{thm}{Theorem}[section]

\numberwithin{equation}{section}

\begin{document}
\vskip1.5cm

\centerline {\bf \chuto Generalized Convex Functions and Their Applications  }

\vskip.2cm

\centerline {\bf \chuto  }


\vskip.8cm \centerline {\kamy  Adem Kili\c{c}man $^{a,}$\footnote{{\tt Corresponding author email: akilic@upm.edu.my  ( Adem Kili\c{c}man)}} and Wedad Saleh $^b$  }

\vskip.5cm
\centerline {$^a$  Institute for Mathematical Research, University Putra Malaysia, Malaysia}
\centerline {e-mail : {\tt akilic@upm.edu.my}}


\centerline {$^b$ Department of Mathematics, Putra University of Malaysia, Malaysia} \centerline {e-mail : {\tt wed\_10\_777@hotmail.com}}




\vskip.5cm \hskip-.5cm{\small{\bf Abstract :}  This study focuses on convex functions and their generalized. Thus, we start this study by giving the definition of convex functions and some of their properties  and discussing a simple geometric property. Then we generalize E-convex functions and establish  some their properties. Moreover, we give generalized $ s $-convex functions in the second sense and present some new inequalities of generalized Hermite-Hadamard type for the class of functions whose second local fractional derivatives of order $ \alpha $ in absolute value at certain powers are generalized $ s $-convex functions in the second sense. At the end, some examples that these inequalities are able to be applied to some special means are showed.



\hrulefill


\section{Introduction}
\hskip0.6cm
Let $ M\subseteq \mathbb{R} $ be an interval. A function $ \varphi:M\subseteq \mathbb{R}\longrightarrow \mathbb{R} $ is called a convex if for any $ y_{1}, y_{2}\in M $ and $ \eta\in [0,1] $,
\begin{equation}\label{c}
\varphi(\eta y_{1}+(1-\eta)y_{2})\leq \eta \varphi(y_{1})+(1-\eta)\varphi(y_{2}).
\end{equation}
If the inequality (\ref{c}) is the strict inequality, then $ \varphi $ is called a strict convex function.\\

From a geometrical point of view, a function $ \varphi $ is convex provided that the line segment connecting any two points of its graph lies on or above the graph. The function $ \varphi $ is strictly convex provided that the line segment connecting any two points of its graph lies above the graph. If $ -\varphi $ is convex (resp. strictly convex), then $ \varphi $ is called concave (resp. strictly concave).\\

The convexity of functions have been widely used in many branches of mathematics, for example in mathematical analysis, function theory, functional analysis, optimization theory and so on. For aproduction function $ x=\varphi(L) $, concacity of $ \varphi $ is expressed economically by saying that $ \varphi $ exhibits diminishing returns. While if $ \varphi $ is convex, then  it exhibits increasing returns.
Due to its applications and significant importance, the concept of convexity has been extended and generalized in several directions, see(\cite{kS,GrinalattLinnainmaa2011,RuelAyres1999}.

\vspace{4.0 mm}
 Recently, the fractal theory has received significantly remarkable attention from scientists and engineers. In the sense of Mandelbrot, a fractal set is the one whose Hausdorff dimension strictly exceeds the topological dimension\cite{Edgar1998,KolwankarGangal1999}. Many researchers studied the properties of functions on fractal space and constructed many kinds of fractional calculus by using different approaches \cite{kS,CarpinteriChiaiaCornetti2001,ZhaoChengYang2013}. Particularly, in \cite{Yang2012}, Yang stated the analysis of local fractional functions on fractal space systematically, which includes local fractional calculus and the monotonicity of function.\\

 Throughout this chapter $ \mathbb{R}^{\alpha} $ will be denoted a real linear fractal set.
 \begin{definition}\cite{HuixiaSuiDongyan2014}
 	A function $ \varphi\colon M\subset \mathbb{R}\longrightarrow \mathbb{R}^{\alpha} $ is called generalized convex if
 	\begin{equation}\label{GC1}
 	\varphi(\eta y_{1}+(1-\eta)y_{2})\leq \eta^{\alpha}\varphi(y_{1})+(1-\eta)^{\alpha}\varphi(y_{2})
 	\end{equation}
  \end{definition}
 	for all $ y_{1},y_{2}\in M $, $ \eta\in[0,1] $ and $ \alpha\in\left(\left. 0,1 \right]  \right.   $.

 \vspace{4.0 mm}
 It is called strictly generalized convex if the inequality (\ref{GC1}) holds strictly whenever $ y_{1} $ and $y_{2} $ are distinct points and $ \eta\in(0,1) $. If $ -\varphi $ is generalized convex (respectively, strictly generalized convex), then $ \varphi$ is generalized concave (respectively, strictly generalized concave).

 \vspace{4.0 mm}
 In $ \alpha=1 $, we have a convex function ,i.e, (\ref{c}) is obtained.

 \vspace{4.0 mm}
 Let $ f\in \,\  _{a_{1}}I^{(\alpha)}_{a_{2}} $ be a generalized convex function on $ [a_{1},a_{2}] $ with $ a_{1}< a_{2} $. Then,

 \begin{equation}\label{eq8}
 f\left( \frac{a_{1}+a_{2}}{2}\right) \leq \frac{\Gamma (1+\alpha)}{(a_{2}-a_{1})^{\alpha}}\,\ _{a_{1}}I^{(\alpha)}_{a_{2}}f(x)\leq \frac{f(a_{1})+f(a_{2})}{2^{\alpha}}.
 \end{equation}
 is known as generalized Hermite-Hadmard's inequality \cite{HuixiaSui2014}.
 Many authers paid attention to the study of generalized Hermite-Hadmard's inequality and generalized convex function, see \cite{AdemWedadNotions2015, Huixia2015Hermite}. If $ \alpha=1 $ in (\ref{eq8}), then \cite{DF}
 \begin{equation}\label{eq9}
 f\left( \frac{a_{1}+a_{2}}{2}\right) \leq\frac{1}{a_{2}-a_{1}}\int_{a_{1}}^{a_{2}}f(x)dx\leq\frac{f(a_{1})+f(a_{2})}{2},
 \end{equation}
 which is known as classical Hermite-Hadamard inequality, for more properties about this inequality we refer the interested readers to \cite{D,hua2014inequalities}.
\section{Generalized E-convex Functions}
\vspace{1.0mm}
In 1999, Youness \cite{Youness1999}  introduced E-convexity of sets and functions, which  have some important applications in various branches of mathematical sciences \cite{Abou1999inequalities,Noor1994fuzzy}. However, Yang \cite{Yang2001} showed that  some results given by Youness  \cite{Youness1999} seem to be incorrect. Chen \cite{Chen2002} extended E-convexity to a semi E-convexity and discussed some of its properties. For more results on E-convex function or semi E-convex function see  \cite{AW,FulgaPreda,Iqbal,IAA2012,SyauLee2005}.
\hskip0.6cm
\begin{definition}\cite{Youness1999}
	\begin{enumerate}
		\item[(i)] A set $ B\subseteq \mathbb{R}^{n} $ is called a E-convex  iff there exists  $ E\colon \mathbb{R}^{n}\longrightarrow \mathbb{R}^{n} $ such that $$ \eta E(r_{1})+(1-\eta)E(r_{2})\in B, \forall r_{1},r_{2}\in B,\eta\in[0,1] .$$
		\item[(ii)]A function $ g\colon \mathbb{R}^{n}\longrightarrow \mathbb{R} $ is called E-convex (ECF) on a set $ B \subseteq \mathbb{R}^{n} $ iff there exists  $ E\colon \mathbb{R}^{n}\longrightarrow \mathbb{R}^{n} $ and $$g(\eta E(r_{1})+(1-\eta) E(r_{2})\leq \eta g(E(r_{1}))+(1-\eta)g(E(r_{2})),\forall r_{1},r_{2}\in B, \eta\in[0,1].$$
	\end{enumerate}
\end{definition}
The following propositions were proved in \cite{Youness1999}:
\begin{proposition}\label{generalizedEconvex3}
	\begin{enumerate}
		\item[(i)] Suppose that a set $ B\subseteq \mathbb{R}^{n} $ is E-convex, then $ E(B)\subseteq B $.
		\item[(ii)] Assume that $ E(B) $ is convex and $ E(B)\subseteq B $, then $ B $ is E-convex.
	\end{enumerate}
\end{proposition}
\begin{definition}
	A function $ g\colon \mathbb{R}^{n}\longrightarrow \mathbb{R}^{\alpha} $ is called a generalized E-convex function (gECF) on a set $ B\subseteq \mathbb{R}^{n} $ iff there exists a map $ E\colon \mathbb{R}^{n}\longrightarrow \mathbb{R}^{n} $ such that $ B $ is an E-convex set and
	\begin{equation}\label{generalizedE-convex1}
	g(\eta E(r_{1})+(1-\eta)E(r_{2}))\leq \eta^{\alpha}g(E(r_{1}))+(1-\eta)^{\alpha}g(E(r_{2})),
	\end{equation}
	$\forall r_{1},r_{2}\in B, \eta\in(0,1) $ and $ \alpha\in\left( \left.0,1 \right] \right.  $
	On the other hand, if $$g(\eta E(x_{1})+(1-\eta)E(x_{2}))\geq \eta^{\alpha}g(E(x_{1}))+(1-\eta)^{\alpha}g(E(x_{1})),$$  $\forall x_{1},x_{2}\in B, \eta\in(0,1) $ and $ \alpha\in\left( \left.0,1 \right] \right.  $, then $ g $ is called generalized E-concave on $ B $. If the inequality sings in the previous two inequality are strict, then $ g $ is called generalized strictly E-convex and generalized strictly E-concave, respectively.
\end{definition}
\begin{proposition}
	\begin{enumerate}
		\item[(i)] Every  ECF  on a convex set $ B $ is gECF , where $ E=I $.
		\item[(ii)] If $ \alpha=1 $ in equation (\ref{generalizedE-convex1}), then $ g $ is called ECF on a set $ B $.
		\item[(iii)]If $ \alpha=1 $ and $ E=I $  in equation (\ref{generalizedE-convex1}), then g is called a convex function
	\end{enumerate}
\end{proposition}	
The following two examples show that generalized E-convex function which are not necessarily generalized convex.
\begin{example}
	Assume that $ B\subseteq \mathbb{R}^{2} $ is given as
	$$B=\left\lbrace (x_{1},x_{2})\in \mathbb{R}^{2}\colon\mu_{1}(0,0)+\mu_{2}(0,3)+\mu_{3}(2,1) \right\rbrace, $$
	with $ \mu_{i}>0 $,$\displaystyle \sum_{i=1}^{3}\mu_{i}=1 $ and define a map $ E\colon \mathbb{R}^{2}\longrightarrow \mathbb{R}^{2} $ such as $ E(x_{1},x_{2})=(0,x_{2}) $. The function $ g\colon \mathbb{R}^{2}\longrightarrow \mathbb{R}^{\alpha} $ defined by
	\begin{eqnarray*}
		g(x_{1},x_{2}) &=& \begin{cases}
			x_{1}^{3\alpha} ; x_{2}< 1,\\
			x_{1}^{\alpha}x_{2}^{3\alpha}; x_{2}\geq 1
		\end{cases}
	\end{eqnarray*}
	The function $ g $ is gECF on B, but is not generalized convex.
\end{example}

\begin{remark}
	If $ \alpha\longrightarrow 0 $ in the above example, then $ g $ goes to generalized convex function.
\end{remark}
\begin{example}
	Assume that $ g\colon \mathbb{R}\longrightarrow \mathbb{R}^{\alpha} $ is defined as
	\begin{eqnarray*}
		g(r) &=& \begin{cases}
			1^{\alpha} ; r> 0,\\
			(-r)^{\alpha}; r\leq 0
		\end{cases}
	\end{eqnarray*}
	and assume that $ E\colon \mathbb{R}\longrightarrow \mathbb{R} $ is defined as $ E(r)=-r^{2} $. Hence, $ \mathbb{R} $ is an E-convex set and $ g $ is gECF, but is not generalized convex.
\end{example}
\begin{thm}
	Assume that $ B\subseteq \mathbb{R}^{n}  $ is an E-convex set and $ g_{1}\colon B\longrightarrow \mathbb{R} $ is an ECF. If $ g_{2}\colon U\longrightarrow \mathbb{R}^{\alpha} $ is  non-decreasing generalized convex function such that the rang $ g_{1} \subset U $, then $ g_{2}og_{1} $ is a gECF on $ B $.
\end{thm}
\begin{proof}
	Since $ g_{1} $ is ECF, then
	$$g_{1}(\eta E(r_{1})+(1-\eta)E(r_{2}))\leq \eta g_{1}(E(r_{1}))+(1-\eta)g_{1}(E(r_{2})), $$
	$ \forall r_{1},r_{2}\in B $ and $ \eta\in[0,1] $. Also, since $ g_{2} $ is  non-decreasing generalized convex function, then
	\begin{eqnarray}
	g_{2}og_{1}(\eta E(r_{1})+(1-\eta)E(r_{2}))&\leq& g_{2}\left[\eta g_{1}(E(r_{1}))+(1-\eta)g_{1}(E(r_{2})) \right]\nonumber\\&\leq& \eta^{\alpha}g_{2}\left(g_{1}(E(r_{1})) \right)+(1-\eta)^{\alpha} g_{2} \left(g_{1}(E(r_{2})) \right)\nonumber\\ &=& \eta^{\alpha}g_{2}og_{1}(E(r_{1}))+(1-\eta)^{\alpha} g_{2} og_{1}(E(r_{2}))\nonumber
	\end{eqnarray}
	which implies that $ g_{2}og_{1} $ is a gECF on $ B $. \\
	
	Similarly, $ g_{2}og_{1} $ is a strictly gECF if $ g_{2} $ is a strictly non-decreasing generalized convex function.
\end{proof}
\begin{thm}
	Assume that $ B\subseteq \mathbb{R}^{n}  $ is an E-convex set, and $ g_{i}\colon B\longrightarrow \mathbb{R}^{\alpha}, i=1,2,...,l $ are generalized E-convex function. Then, $$g=\sum_{i=1}^{l}k_{i}^{\alpha}g_{i}$$ is a generalized E-convex on $ B $ for all $ k_{i}^{\alpha}\in \mathbb{R}^{\alpha} $
\end{thm}
\begin{proof}
	Since $ g_{i},i=1,2,...,l $ are gECF, then $$g_{i}(\eta E(r_{1})+(1-\eta)E(r_{2}))\leq \eta^{\alpha} g_{i}(E(r_{1}))+(1-\eta)^{\alpha}g_{i}(E(r_{2})), $$
	$ \forall r_{1},r_{2}\in B $ , $ \eta\in[0,1] $ and $ \alpha\in\left(\left.  0,1 \right]\right.   $.
	Then,
	\begin{eqnarray}&&
	\sum_{i=1}^{l}k_{i}^{\alpha}g_{i}(\eta E(r_{1})+(1-\eta)E(r_{2}))\nonumber\\&&\hspace{0.5in} \leq \eta^{\alpha}\sum_{i=1}^{l}k_{i}^{\alpha}g_{i}(E(r_{1}))+
	(1-\eta)^{\alpha}\sum_{i=1}^{l}k_{i}^{\alpha}g_{i}(E(r_{2}))\nonumber\\&&\hspace{0.5in}= \eta^{\alpha}g(E(r_{1}))+(1-\eta)^{\alpha}g(E(r_{2})).\nonumber
	\end{eqnarray}
	Thus, $ g $ is a gECF.
\end{proof}	
\begin{definition}
	Assume that $ B\subseteq \mathbb{R}^{n} $ is a convex set. A function $ g\colon B\longrightarrow \mathbb{R}^{\alpha} $ is called generalized quasi convex if $$g(\eta r_{1}+(1-\eta)r_{2})\leq \max\left\lbrace g(r_{1}), g(r_{2}) \right\rbrace, $$
	$ \forall r_{1},r_{2}\in B $ and $ \eta\in[0,1] $.
\end{definition}
\begin{definition}
	Assume that $ B\subseteq \mathbb{R}^{n} $ is an E-convex set. A function  $ g\colon B\longrightarrow \mathbb{R}^{\alpha} $ is called
	\begin{itemize}
		\item[(i)] Generalized E-quasiconvex function iff $$g(\eta E(r_{1})+(1-\eta)E(r_{2})\leq \max\left\lbrace g(E(r_{1})), E(g(r_{2})) \right\rbrace, $$
		$ \forall r_{1},r_{2}\in B $ and $ \eta\in[0,1] $.
		\item[(ii)]Strictly generalized E-quasiconcave function iff
		$$g(\eta E(r_{1})+(1-\eta)E(r_{2})> \min\left\lbrace g(E(r_{1})), E(g(r_{2})) \right\rbrace, $$
		$ \forall r_{1},r_{2}\in B $ and $ \eta\in[0,1] $.
	\end{itemize}
\end{definition}
\begin{thm}\label{generalizedEconvex2}
	Assume that $ B\subseteq \mathbb{R}^{n}  $ is an E-convex set, and $ g_{i}\colon B\longrightarrow \mathbb{R}^{\alpha}, i=1,2,...,l $ are gECF. Then,
	\begin{itemize}
		\item[(i)] The function $ g\colon B\longrightarrow \mathbb{R}^{\alpha} $ which is defined by $ g(r)=\sup_{i\in I}g_{i}(r),r\in B $ is a gECF on $ B $.
		\item[(ii)]If $ g_{i},i=1,2,...,l $  are generalized E-quasiconvex functions on $ B $, then the function $ g $ is a generalized E-quasiconvex function on $ B $.
	\end{itemize}
\end{thm}
\begin{proof}
	\begin{itemize}
		\item[(i)] Due to  $ g_{i}, i\in I $ be gECF on $ B $, then
		\begin{eqnarray}&&
		g(\eta E(r_{1})+(1-\eta)E(r_{2}))\nonumber\\&&\hspace{0.5in} = \sup_{i\in I}g_{i}(\eta E(r_{1})+(1-\eta)E(r_{2}))\nonumber\\&&\hspace{0.5in} \leq  \eta^{\alpha}\sup_{i\in I} g_{i}(E(r_{1}))+(1-\eta)^{\alpha}\sup_{i\in I} g_{i}(E(r_{2}))\nonumber\\&&\hspace{0.5in}= \eta^{\alpha}g(E(r_{1}))+(1-\eta)^{\alpha}g(E(r_{2})).\nonumber
		\end{eqnarray}
		Hence, $ g $ is a gECF on $ B $.
		\item[(ii)] Since $ g_{i} , i\in I $ are generalized E-quasiconvex functions on $ B $, then
		\begin{eqnarray}
		g(\eta E(x_{1})+(1-\eta)E(x_{2})) &=& \sup_{i\in I}g_{i}(\eta E(x_{1})+(1-\eta)E(x_{2}))\nonumber\\ &\leq& \sup_{i\in I} \max \left\lbrace g_{i}(E(x_{1})),g_{i}(E(x_{2})) \right\rbrace\nonumber\\ &=&   \max \left\lbrace\sup_{i\in I} g_{i}(E(x_{1})),\sup_{i\in I} g_{i}(E(x_{2})) \right\rbrace\nonumber\\ &=& \max \left\lbrace g(E(x_{1})), g(E(x_{2})) \right\rbrace.\nonumber
		\end{eqnarray}
		Hence, $ g $ is a generalized E-quasiconvex function on $ B $.
	\end{itemize}
\end{proof}
Considering $ B\subseteq \mathbb{R}^{n} $ is a nonempty E-convex set.  From Propostion\ref{generalizedEconvex3}(i), we get $ E(B)\subseteq B $.Hence, for any $ g\colon B\longrightarrow \mathbb{R}^{\alpha} $, the restriction $ \tilde{g}\colon E(B)\longrightarrow \mathbb{R}^{\alpha} $ of g to E(B)  defined by
$$\tilde{g}(\tilde{x}) =g(\tilde{x}),\forall\tilde{x}\in E(B)$$ is well defined.
\begin{thm}\label{generalizedEconvex4}
	Assume that $ B\subseteq \mathbb{R}^{n} $, and $ g\colon B\longrightarrow \mathbb{R}^{\alpha} $ is a generalized E-quasiconvex function on $ B $. Then, the restriction $ \tilde{g}\colon U\longrightarrow \mathbb{R}^{\alpha} $ of $ g $ to any nonempty convex subset $ U $ of $ E(B) $ is a generalized quasiconvex on $ U $.
\end{thm}
\begin{proof}
	Assume that $ x_{1}, x_{2}\in U\subseteq E(B) $, then there exist $ x_{1}^{*},x_{2}^{*}\in B $ such that $ x_{1}=E(x_{1}^{*}) $ and $ x_{2}=E(x_{2}^{*}) $. Since $ U $ is a convex set, we have $$\eta x_{1}+(1-\eta)x_{2}=\eta E(x^{*}_{1})+(1-\eta)E(x^{*}_{2})\in U, \forall\eta\in[0,1].$$Therefore, we have
	\begin{eqnarray}
	\tilde{g}(\eta x_{1}+(1-\eta)x_{2})&=& \tilde{g}(\eta E(x^{*}_{1})+(1-\eta)E(x^{*}_{2}))\nonumber\\ &\leq& \max\left\lbrace g(E(x^{*}_{1})),g(E(x^{*}_{2}))\right\rbrace \nonumber\\ &=& \max \left\lbrace g(x_{1}),g(x_{2})\right\rbrace\nonumber\\ &=& \max \left\lbrace \tilde{g}(x_{1}),\tilde{g}(x_{2})\right\rbrace.\nonumber
	\end{eqnarray}
\end{proof}
\begin{thm}\label{generalizedEconvex5}
	Assume that $ B\subseteq \mathbb{R}^{n} $ is an E-convex set, and $ E(B) $ is a convex set. Then, $ g\colon B\longrightarrow \mathbb{R}^{\alpha} $ is a generalized E-quasiconvex on $ B $ iff its restriction $ \tilde{g}=g_{|_{E(B)}} $ is a generalized quasiconvex function on $ E(B) $.
\end{thm}
\begin{proof}
	Due to Theorem\ref{generalizedEconvex4}, the if condition is true. Conversely, suppose that $ x_{1},x_{2}\in B $, then $ E(x_{1}),E(x_{2}) \in E(B) $ and $ \eta E(x_{1}) +(1-\eta)E(x_{2})\in E(B) \subseteq B, \forall\eta\in[0,1] $.\\
	Since $ E(B)\subseteq B $, then
	\begin{eqnarray}
	g(\eta E(x_{1}) +(1-\eta)E(x_{2}))&=& \tilde{g}(\eta E(x_{1}) +(1-\eta)E(x_{2}))\nonumber\\ &\leq& \max\left\lbrace\tilde{g}(E(x_{1})),\tilde{g}(E(x_{2})) \right\rbrace\nonumber\\ &=& \max\left\lbrace g(E(x_{1})),g(E(x_{2})) \right\rbrace . \nonumber
	\end{eqnarray}
\end{proof}
An analogous result to Theorem\ref{generalizedEconvex4} for the generalized E-convex case is as follows:
\begin{thm}\label{generalizedEconvex6}
	Assume that $ B\subseteq \mathbb{R}^{n} $ is an E-convex set, and $ g\colon B\longrightarrow \mathbb{R}^{\alpha} $ is a gECF on $ B $. Then, the restriction $ \tilde{g}\colon U\longrightarrow \mathbb{R}^{\alpha} $ of $ g $ to any nonempty convex subset $ U $ of $ E(B) $ is a gCF.
\end{thm}
An analogous result to Theorem\ref{generalizedEconvex5} for the generalized E-convex case is as follows:
\begin{thm}
	Assume that $ B\subseteq \mathbb{R}^{n} $ is an E-convex set, and $ E(B) $ is a convex set. Then, $ g\colon B\longrightarrow \mathbb{R}^{\alpha} $ is a gECF on $ B $ iff its restriction $ \tilde{g}=g_{|_{E(B)}} $ is a gCF on $ E(B) $.
\end{thm}
The lower level set of $ goE\colon B\longrightarrow \mathbb{R}^{\alpha} $ is defined as
$$L_{r^{\alpha}}(goE)=\left\lbrace x\in B\colon (goE)(x)=g(E(x))\leq r^{\alpha},r^{\alpha}\in \mathbb{R}^{\alpha} \right\rbrace. $$
The lower level set of $ \tilde{g}\colon E(B)\longrightarrow \mathbb{R}^{\alpha} $ is defined as
$$L_{r^{\alpha}}(\tilde{g})=\left\lbrace\tilde{x}\in E(B)\colon \tilde{g}(\tilde{x})=g(\tilde{x})\leq r^{\alpha}, r^{\alpha}\in\mathbb{R}^{\alpha} \right\rbrace .$$
\begin{thm}
	Supose that $ E(B) $ be a convex set. A function $ g\colon B\longrightarrow \mathbb{R}^{\alpha} $ is a generalized E-quasiconvex iff $ L_{r^{\alpha}}(\tilde{g}) $ of its restriction $ \tilde{g}\colon E(B)\longrightarrow \mathbb{R}^{\alpha}  $ is a convex set for each $ r^{\alpha}\in \mathbb{R}^{\alpha} $.
\end{thm}
\begin{proof}
	Due to  $ E(B) $ be a convex set, then for each $ E(x_{1}),E(x_{2})\in E(B) $, we have $ \eta E(x_{1})+(1-\eta) E(x_{2})\in E(B)\subseteq B$. Let $ \tilde{x}_{1}=E(x_{1}) $ and $ \tilde{x}_{2}=E(x_{2}) $. If $ \tilde{x}_{1},\tilde{x}_{2}\in L_{r^{\alpha}}(\tilde{g}) $, then $ g(\tilde{x}_{1})\leq r^{\alpha} $ and $ g(\tilde{x}_{2})\leq r^{\alpha} $. Thus,
	\begin{eqnarray}
	\tilde{g}(\eta \tilde{x}_{1}+(1-\eta)\tilde{x}_{2})&=& g(\eta \tilde{x}_{1}+(1-\eta)\tilde{x}_{2})\nonumber\\ &=& g(\eta E(x_{1})+(1-\eta)E(x_{2}))\nonumber\\ &\leq& \max\left\lbrace g(E(x_{1})) ,g(E(x_{2}))\right\rbrace\nonumber\\ &=& \max \left\lbrace g(\tilde{x}_{1}),g(\tilde{x}_{2}) \right\rbrace \nonumber\\ &=& \max\left\lbrace \tilde{g}(\tilde{x}_{1}),\tilde{g}(\tilde{x}_{2})\right\rbrace\nonumber\\ &\leq& r^{\alpha}.\nonumber
	\end{eqnarray}
	which show that $ \eta \tilde{x}_{1}+(1-\eta) \tilde{x}_{2}\in L_{r^{\alpha}}(\tilde{g}) $. Hence, $ L_{r^{\alpha}}(\tilde{g}) $ is a convex set.\\
	
	Conversely, let $ L_{r^{\alpha}}(\tilde{g}) $ be a convex set for each $ r^{\alpha}\in \mathbb{R}^{\alpha} $, i.e., $ \eta \tilde{x}_{1}+(1-\eta) \tilde{x}_{2}\in L_{r^{\alpha}}(\tilde{g}), \forall \tilde{x}_{1},\tilde{x}_{2}\in L_{r^{\alpha}}(\tilde{g}) $ and $ r^{\alpha}=\max\left\lbrace g(\tilde{x}_{1}),g(\tilde{x}_{2}) \right\rbrace $. Thus,
	\begin{eqnarray}
	g(\eta E(x_{1})+(1-\eta)E(x_{2})) &=& \tilde{g}(\eta E(x_{1})+(1-\eta)E(x_{2}))\nonumber\\ &=& \tilde{g}(\eta \tilde{x}_{1}+(1-\eta)\tilde{x}_{2})\nonumber\\ &\leq& r^{\alpha}\nonumber\\&=&\max\left\lbrace g(\tilde{x}_{1}),g(\tilde{x}_{2}) \right\rbrace \nonumber\\ &=& \max\left\lbrace g(E(x_{1})),g(E(x_{2})) \right\rbrace. \nonumber
	\end{eqnarray}
	Hence, $ g $ is a generalized E-quasiconvex.
\end{proof}
\begin{thm}
	Let $ B\subseteq \mathbb{R}^{n} $ be a nonempty E-convex set and let $ g_{1}\colon B\longrightarrow \mathbb{R}^{\alpha} $ be a generalized E-quasiconvex on $ B $. Suppose that $ g_{2}\colon \mathbb{R}^{\alpha}\longrightarrow \mathbb{R}^{\alpha} $ is a non-decreasing function. Then, $ g_{2}og_{1} $ is a generalized E-quasiconvex.
\end{thm}
\begin{proof}
	Since  $ g_{1}\colon B\longrightarrow \mathbb{R}^{\alpha} $ is generalized E-quasiconvex on $ B $ and $ g_{2}\colon \mathbb{R}^{\alpha}\longrightarrow \mathbb{R}^{\alpha} $ is a non-decreasing function, then
	\begin{eqnarray}
	(g_{2}og_{1})(\eta E(x_{1})+(1-\eta)E(x_{2}))&=& g_{2}(g_{1}(\eta E(x_{1})+(1-\eta)E(x_{2})))\nonumber\\ &\leq& g_{2}(\max\left\lbrace g_{1}(E(x_{1})),g_{1}(E(x_{2})) \right\rbrace )\nonumber\\ &=& \max\left\lbrace (g_{2}og_{1})(E(x_{1})),(g_{2}og_{1})(E(x_{2})) \right\rbrace  \nonumber
	\end{eqnarray}
	which shows that $ g_{2}og_{1} $ is a generalized E-quasiconvex on $ B $.
\end{proof}
\newpage
\begin{thm}
	If the function $ g $ is a gECF on $ B\subseteq \mathbb{R}^{n} $, then $ g $ is a generalized E-quasiconvex on $ B $.
\end{thm}
\begin{proof}
	Assume that $ g $ is a gECF on $ B $. Then,
	\begin{eqnarray}
	g(\eta E(r_{1})+(1-\eta)E(r_{2}))&\leq& \eta^{\alpha}g(E(r_{1}))+
	(1-\eta)^{\alpha}g(E(r_{2}))\nonumber\\&\leq& \eta^{\alpha}\max\left\lbrace g(E(r_{1})),g(E(r_{2})) \right\rbrace\nonumber\\\hspace{0.5in}&&  +(1-\eta)^{\alpha}\max\left\lbrace g(E(r_{1})),g(E(r_{2})) \right\rbrace\ \nonumber\\&=& \max\left\lbrace g(E(r_{1})),g(E(r_{2})) \right\rbrace.\nonumber
	\end{eqnarray}
\end{proof}
\section{$ E^{\alpha} $-epigraph}
\begin{definition}
	Assume that $ B\subseteq \mathbb{R}^{n}\times \mathbb{R}^{\alpha} $ and $ E\colon \mathbb{R}^{n}\longrightarrow \mathbb{R}^{n} $, then the set B is called $ E^{\alpha} $-convex set iff $$\left(\eta E(x_{1})+(1-\eta)E(x_{2}), \eta^{\alpha}r^{\alpha}_{1}+(1-\eta)^{\alpha}r^{\alpha}_{2} \right)\in B $$
	$ \forall (x_{1},r^{\alpha}_{1}),(x_{2},r^{\alpha}_{2})\in B $, $ \eta\in[0,1] $ and $ \alpha\in\left(\left.0,1 \right]  \right. $.
\end{definition}
Now, the $ E^{\alpha} $- epigraph of $ g $ is given by
$$epi_{E^{\alpha}}(g)=\left\lbrace(E(x),r^{\alpha})\colon x\in B, r^{\alpha}\in \mathbb{R}^{\alpha}, g(E(x))\leq r^{\alpha} \right\rbrace. $$

A sufficient condition for $ g $ to be a gECF is given by the following theorem:
\begin{thm}
	Let $ E\colon \mathbb{R}^{n}\longrightarrow \mathbb{R}^{n} $ be an idempoted map. Assume that $ B\subseteq \mathbb{R}^{n}$ is an E-convex set and $ epi_{E^{\alpha}}(g) $ is an $ E^{\alpha} $-convex set where $ g\colon B\longrightarrow \mathbb{R}^{\alpha} $, then $ g $ is a gECF on $ B $.
\end{thm}
\begin{proof}
	Assume that $ r_{1},r_{2}\in B $ and $ (E(r_{1}),g(E(r_{1}))),(E(r_{2}),g(E(r_{2})))\in epi_{E^{\alpha}}(g) $. Since $ epi_{E^{\alpha}}(g) $ is $ E^{\alpha} $- convex set, we have
	$$\left(\eta E(E(r_{1}))+(1-\eta)E(E(r_{2})),\eta^{\alpha}g(E(r_{1}))+(1-\eta)^{\alpha}g(E(r_{2})) \right)\in epi_{E^{\alpha}}(g) ,$$
	then
	$$g(E(\eta E(r_{1}))+(1-\eta)E(E(r_{2})))\leq \eta^{\alpha}g(E(r_{1}))+(1-\eta)^{\alpha}g(E(r_{2})). $$
	Due to $ E\colon \mathbb{R}^{n}\longrightarrow \mathbb{R}^{n} $ be an idempotent map, then
	$$g(\eta E(r_{1})+(1-\eta)E(r_{2}))\leq \eta^{\alpha}g(E(r_{1}))+(1-\eta)^{\alpha}g(E(r_{2})). $$
	Hence, $ g $ is a gECF.
\end{proof}
\begin{thm}\label{generalizedEconvex7}
	Assume that $ \left\lbrace B_{i} \right\rbrace _{i\in I} $ is a family of $ E^{\alpha} $-convex sets. Then, their intersection $ \cap_{i\in I}B_{i} $ is an $ E^{\alpha} $-convex set.	
\end{thm}
\begin{proof}
	Considering $ (x_{1}, r_{1}^{\alpha}),(x_{2}, r_{2}^{\alpha})\in \cap_{i\in I}B_{i} $, then $ (x_{1}, r_{1}^{\alpha}),(x_{2}, r_{2}^{\alpha})\in B_{i}$, $\forall i\in I $.  By $ E^{\alpha} $-convexity of $ B_{i},\forall i\in I $, then we have
	$$\left(\eta E(x_{1})+(1-\eta) E(x_{2}), \eta^{\alpha}r_{1}^{\alpha}+(1-\eta)^{\alpha}r^{\alpha}_{2} \right) \in B_{i}, $$ $ \forall \eta\in[0,1] $ and $\alpha\in\left(\left.0,1 \right]  \right.$. Hence,  $$\left(\eta  E(x_{1})+(1-\eta) E(x_{2}), \eta^{\alpha}r_{1}^{\alpha}+(1-\eta)^{\alpha}r^{\alpha}_{2} \right) \in \cap_{i\in I}B_{i} .$$
\end{proof}
The following theorem is a special case of Theorem \ref{generalizedEconvex2}(i) where $ E\colon \mathbb{R}^{n}\longrightarrow \mathbb{R}^{n}  $ is an idempotent map.
\begin{thm}
	Assume that $ E\colon \mathbb{R}^{n}\longrightarrow \mathbb{R}^{n}  $ is an idempotent map, and $ B\subseteq \mathbb{R}^{n} $ is an E-convex set. Let $ \left\lbrace g_{i} \right\rbrace _{i\in I} $ be a family function which have bounded from above. If $ epi_{E^{\alpha}} (g_{i})$ are $ E^{\alpha} $-convex sets, then the function $ g $ which defined by $ g(x)= \sup_{i\in I}g_{i}(x), x\in B $ is a gECF  on $ B $.
\end{thm}
\begin{proof}
	Since
	$$epi_{E^{\alpha}} (g_{i})=\left\lbrace (E(x), r^{\alpha})\colon x\in B, r^{\alpha}\in\mathbb{R}^{\alpha}, g_{i}(E(x))\leq r^{\alpha},i\in I \right\rbrace $$ are $ E^{\alpha} $- convex set in $ B\times \mathbb{R}^{\alpha} $, then
	\begin{eqnarray}
	\cap_{i\in I}epi_{E^{\alpha}} (g_{i})&=&\left\lbrace (E(x), r^{\alpha})\colon x\in B, r^{\alpha}\in\mathbb{R}^{\alpha}, g_{i}(E(x))\leq r^{\alpha}, i\in I \right\rbrace \nonumber\\ &=& \left\lbrace (E(x), r^{\alpha})\colon x\in B, r^{\alpha}\in\mathbb{R}^{\alpha}, g(E(x))\leq r^{\alpha} \right\rbrace,
	\end{eqnarray}
	where $ g(E(x))=\sup_{i\in I}g_{i}(E(x)) $, also is $ E^{\alpha} $-convex set. Hence, $ \cap_{i\in I} epi_{E^{\alpha}} (g_{i}) $ is an $ E^{\alpha} $-epigraph, then by Theorem\ref{generalizedEconvex7}, $ g $ is a generalized E-convex function on $ B $.
\end{proof}
\section{Generalized $ s $-convex functions}
\hskip0.6cm
There are many researchers studied the properties of functions on fractal space and constructed many kinds of fractional calculus by using different approaches see \cite{Baleanu2015Local,  Mehmet2016generalized,Yang2015Local}\\

  In  \cite{HuixiaSui2014},  two kinds of generlized $ s $-convex functions on fractal sets are introduced  as follows:
  \begin{definition}\label{1}
  	\begin{enumerate}
  		\item[(i)] A function $ \varphi\colon\mathbb{R}_{+}\longrightarrow\mathbb{R}^{\alpha} $, is called a generalized $ s $-convex  ($ 0<s<1 $)in the first sense if
  		\begin{equation}\label{11}
  		\varphi(\eta_{1}y_{1}+\eta_{2}y_{2})\leq \eta_{1}^{s\alpha}\varphi(y_{1})+\eta_{2}^{s \alpha}\varphi(y_{2})
  		\end{equation}
  		for all $ y_{1},y_{2}\in \mathbb{R}_{+} $ and all $ \eta_{1},\eta_{2}\geq0 $ with $ \eta_{1}^{s}+\eta_{2}^{s}=1 $, this class of functions is denoted by $ GK^{1}_{s} $ .
  		\item[(ii)] A function $ \varphi\colon\mathbb{R}_{+}\longrightarrow\mathbb{R}^{\alpha} $, is called a  generalized $ s $-convex ($ 0<s<1 $) in the second sense if (\ref{11}) holds for all $ y_{1},y_{2}\in \mathbb{R}_{+} $ and all $ \eta_{1},\eta_{2}\geq0 $ with $ \eta_{1}+\eta_{2}=1 $, this class of functions is denoted by $ GK^{2}_{s} $.
  	\end{enumerate}
  \end{definition}
   In the same paper,\cite{HuixiaSui2014}, Mo and Sui proved that all functions which are generalized $ s $-convex in the second sense, for $ s\in (0,1) $, are non-negative.\\

  If $ \alpha=1 $ in  Definition \ref{1} , then we have the classical $ s $-convex functions in the first sense (second sense) see \cite{DF}.\\

   Also, in \cite{DF}, Dragomir and Fitzatrick demonstrated a variation of Hadamard's inequality which holds for $ s $-convex functions in the second sense.
  \begin{theorem}
  	Assume that $ \varphi\colon \mathbb{R}_{+}\longrightarrow \mathbb{R}_{+} $ is a s-convex function in the second sense, $0< s<1 $ and $ y_{1},y_{2}\in \mathbb{R}_{+}$, $ y_{1}<y_{2} $. If $ \varphi\in L^{1}([y_{1},y_{2}])$, then
  	\begin{equation}\label{eq10}
  	2^{s-1}\varphi\left( \frac{y_{1}+y_{2}}{2}\right) \leq\frac{1}{y_{2}-y_{1}}\int_{y_{1}}^{y_{2}}\varphi(z)dz\leq\frac{\varphi(y_{1})+\varphi(y_{2})}{s+1}.
  	\end{equation}
  \end{theorem}
   If we set  $ k=\dfrac{1}{s+1} $, then it is the best possible in the second inequality in (\ref{eq10}).\\

   A varation of generalized Hadamard's inequality which holds for generalized $ s $-convex functions in the second sense \cite{kS}.
  \begin{theorem}
  	Assume that $ \varphi\colon\mathbb{R}_{+}\longrightarrow\mathbb{R}^{\alpha}_{+} $ is a generalized $ s $-convex function in the second sense where $ 0<s<1  $ and $ y_{1},y_{2}\in\mathbb{R}_{+}  $  with $y_{1}<y_{2} $. If $ \varphi\in L^{1}([y_{1},y_{2}]) $, then
  	\begin{eqnarray}&&\label{eq11}
  		2^{\alpha(s-1)}\varphi\left( \frac{y_{1}+y_{2}}{2}\right) \leq\frac{\Gamma(1+\alpha)}{(y_{2}-y_{1})^{\alpha}}\,\ _{y_{1}}I^{(\alpha)}_{y_{2}} \varphi(x)\nonumber\\&&\hspace{1.3in}\leq\frac{\Gamma(1+s\alpha)\Gamma(1+\alpha)}{\Gamma(1+(s+1)\alpha)}\left( \varphi(y_{1})+\varphi(y_{2})\right).
  	\end{eqnarray}
  \end{theorem}
\begin{proof}
We know that $\varphi $ is  generalized $ s $-convex in the second sense,which lead to
	$$\varphi(\eta y_{1}+(1-\eta)y_{2})\leq \eta^{\alpha s}\varphi(y_{1})+(1-\eta)^{\alpha s}\varphi(y_{2}), \forall \eta\in[0,1].$$
Then the following inequality can be written:
	\begin{eqnarray}
	\Gamma(1+\alpha)\,\ _{0}I^{(\alpha)}_{1} \varphi(\eta y_{1}+(1-\eta)y_{2}) &\leq& \varphi(y_{1})\Gamma(1+\alpha)\,\ _{0}I^{(\alpha)}_{1} \eta^{\alpha s}\nonumber\\&&\hspace{0.3in} + \varphi(y_{2})\Gamma(1+\alpha)\,\ _{0}I^{(\alpha)}_{1} (1-\eta)^{\alpha s}\nonumber\\ &=& \frac{\Gamma(1+s\alpha)\Gamma(1+\alpha)}{\Gamma(1+(s+1)\alpha)}(\varphi(y_{2})+y_{2})\nonumber
	\end{eqnarray}
By	considering $ z=\eta y_{1}+(1-\eta) y_{2}$. Then
	\begin{eqnarray}
	\Gamma(1+\alpha)\,\ _{0}I^{(\alpha)}_{1} \varphi(\eta y_{1}+(1-\eta)y_{2}) &=&  \frac{\Gamma(1+\alpha)}{(y_{1}-y_{2})^{\alpha}}\,\ _{y_{2}}I^{(\alpha)}_{y{1}} \varphi(x)\nonumber\\ &=& \frac{\Gamma(1+\alpha)}{(y_{2}-y_{1})^{\alpha}}\,\ _{y_{1}}I^{(\alpha)}_{y_{2}} \varphi(z).\nonumber
	\end{eqnarray}
	Here
	$$\frac{\Gamma(1+\alpha)}{(y_{2}-y_{1})^{\alpha}}\,\ _{y_{1}}I^{(\alpha)}_{y_{2}} \varphi(x)\leq \frac{\Gamma(1+s\alpha)\Gamma(1+\alpha)}{\Gamma(1+(s+1)\alpha)}(\varphi(y_{1})+\varphi(y_{2})).$$
	Then, the second inequality in (\ref{eq11}) is given.\\
	
Now
	\begin{equation}\label{h4}
	\varphi\left( \frac{z_{1}+z_{2}}{2}\right) \leq\frac{\varphi(z_{1})+\varphi(z_{2})}{2^{\alpha s}},\forall z_{1},z_{2}\in I.
	\end{equation}
	Let $ z_{1}=\eta y_{1}+(1-\eta)y_{2} $ and $ z_{2}=(1-\eta)y_{1}+\eta y_{2} $ with $ \eta\in[0,1] $.\\
	Hence, by applying (\ref{h4}), the next inequalty holds
	$$ \varphi\left( \frac{y_{1}+y_{2}}{2}\right) \leq\frac{\varphi(\eta y_{1}+(1-\eta)y_{2})+\varphi((1-\eta)y_{1}+\eta y_{2})}{2^{\alpha s}} ,\, \forall \eta\in[0,1].$$
	
	So
	$$\frac{1}{\Gamma(1+\alpha)}\int_{0}^{1}\varphi\left( \frac{y_{1}+y_{2}}{2}\right) (d\eta)^{\alpha}\leq \frac{1}{2^{\alpha (s-1)} (y_{2}-y_{1})^{\alpha}}\,\ _{y_{1}}I^{(\alpha)}_{y_{2}}\varphi(z) $$
	Then,
	$$ 2^{\alpha(s-1)} \varphi\left( \frac{y_{1}+y_{2}}{2}\right) \leq \frac{\Gamma(1+\alpha)}{(y_{2}-y_{1})^{\alpha}}\,\ _{y_{1}}I^{(\alpha)}_{y_{2}}\varphi(z).$$
\end{proof}
   \begin{lemma}\label{some1}
   	Assume that $\varphi\colon [y_{1},y_{2}] \subset\mathbb{R}\longrightarrow\mathbb{R}^{\alpha} $ is a local fractional derivative of order $ \alpha $ ($ \varphi\in D_{\alpha} $) on $ (y_{1},y_{2}) $  with $ y_{1}<y_{2} $. If $ \varphi^{(2\alpha)}\in C_{\alpha}[y_{1},y_{2}] $, then the following equality holds:
   	\begin{eqnarray*}
   		\frac{\Gamma(1+2\alpha)[\Gamma(1+\alpha)]^{2}}{2^{\alpha}(y_{2}-y_{1})}\,\ _{y_{1}}I^{(\alpha)}_{y_{2}}\varphi(x)-\frac{\Gamma (1+2\alpha)}{2^{\alpha}}\varphi\left( \frac{y_{1}+y_{2}}{2}\right) \nonumber\\&&\hspace{-3.5in}= \frac{(y_{2}-y_{1})^{2\alpha}}{16^{\alpha}}\left[\,\ _{0}I^{(\alpha)}_{1}\gamma^{2\alpha}\varphi^{(2\alpha)}\left( \gamma \frac{y_{1}+y_{2}}{2}+(1-\gamma)y_{1}\right) \right. \nonumber\\&&\hspace{-3.2in}\left. +\,\ _{0}I^{(\alpha)}_{1}(\gamma-1)^{2\alpha}\varphi^{(2\alpha)}\left( \gamma y_{2}+(1-\gamma)\frac{y_{1}+y_{2}}{2}\right)  \right]
   	\end{eqnarray*}
   \end{lemma}
   \begin{proof}
   	From the local fractional integration by parts, we get
   	\begin{eqnarray*}
   		B_{1}&=& \frac{1}{\Gamma(1+\alpha)}\int_{0}^{1}\gamma^{2\alpha}\varphi^{(2\alpha)}\left( \gamma\frac{y_{1}+y_{2}}{2}+(1-\gamma)y_{1}\right) (d\gamma)^{\alpha}
   		\nonumber\\&&\hspace{-0.3in}= \left( \frac{2}{y_{2}-y_{1}}\right) ^{\alpha}\varphi^{(\alpha)}\left( \frac{y_{1}+y_{2}}{2}\right)\nonumber\\&&\hspace{-0.3in} -\Gamma(1+2\alpha)\left( \frac{2}{y_{2}-y_{1}}\right) ^{2\alpha}\gamma^{\alpha}\left. \varphi\left( \gamma\frac{y_{1}+y_{2}}{2}+(1-\gamma)y_{1}\right) \right| _{0}^{1}\nonumber\\&&\hspace{-0.3in}+\Gamma(1+2\alpha)\Gamma(1+\alpha)\left( \frac{2}{b-a}\right) ^{2\alpha}\int_{0}^{1}\varphi\left( \gamma \frac{y_{1}+y_{2}}{2}+(1-\gamma)y_{1}\right) (d\gamma)^{\alpha}
   		\nonumber\\&&\hspace{-0.3in}=\left( \frac{2}{y_{2}-y_{1}}\right) ^{\alpha}\varphi^{(\alpha)}\left( \frac{y_{1}+y_{2}}{2}\right) -\Gamma(1+2\alpha)\left( \frac{2}{y_{2}-y_{1}}\right) ^{2\alpha}\varphi\left( \frac{y_{1}+y_{2}}{2}\right) \nonumber\\&&\hspace{-0.3in}+\Gamma(1+2\alpha)\Gamma(1+\alpha)\left( \frac{2}{y_{2}-y_{1}}\right) ^{2\alpha}\int_{0}^{1}\varphi\left( \gamma \frac{y_{1}+y_{2}}{2}+(1-\gamma)y_{1}\right) (d\gamma)^{\alpha}
   	\end{eqnarray*}
   	Setting $ x=\gamma\dfrac{y_{1}+y_{2}}{2}+(1-\gamma)y_{1} $, for $ \gamma\in[0,1] $ and multiply the both sides in the last equation by $ \dfrac{(y_{2}-y_{1})^{2\alpha}}{16^{\alpha}} $, we get
   	\begin{eqnarray*}
   		B_{1}&=& \frac{(y_{2}-y_{1})^{2\alpha}}{16^{\alpha}}_{0}I^{(\alpha)}_{1}\gamma^{2\alpha}\varphi^{(2\alpha)}\left( \gamma\frac{y_{1}+y_{2}}{2}+(1-\gamma)y_{1}\right) \nonumber\\&=&\frac{(y_{2}-y_{1})^{\alpha}}{8^{\alpha}}\varphi^{(\alpha)}\left( \frac{y_{1}+y_{2}}{2}\right) -\frac{\Gamma(1+2\alpha)}{4^{\alpha}}\varphi\left( \frac{y_{1}+y_{2}}{2}\right) \nonumber\\&&\hspace{0.5in}+\frac{\Gamma(1+2\alpha)\Gamma(1+\alpha)}{2^{\alpha}(y_{2}-y_{1})^{\alpha}}\int_{y_{1}}^{\frac{y_{1}+y_{2}}{2}}
   		\varphi(x)(dx)^{\alpha}.
   	\end{eqnarray*}
   	By the similar way, also we have
   	\begin{eqnarray*}
   		B_{2}&=& \frac{(y_{2}-y_{1})^{2\alpha}}{16^{\alpha}}_{0}I^{(\alpha)}_{1}(\gamma-1)^{2\alpha}\varphi^{(2\alpha)}\left( \gamma y_{2} +(1-\gamma)\frac{y_{1}+y_{2}}{2}\right) \nonumber\\&=&-\frac{(y_{2}-y_{1})^{\alpha}}{8^{\alpha}}\varphi^{(\alpha)}\left( \frac{y_{1}+y_{2}}{2}\right) -\frac{\Gamma(1+2\alpha)}{4^{\alpha}}\varphi\left( \frac{y_{1}+y_{2}}{2}\right) \nonumber\\&&\hspace{0.5in}+\frac{\Gamma(1+2\alpha)\Gamma(1+\alpha)}{2^{\alpha}(y_{2}-y_{1})^{\alpha}}\int^{y_{2}}_{\frac{y_{1}+y_{2}}{2}}\varphi(x)(dx)^{\alpha}.
   	\end{eqnarray*}
   	Thus, adding $ B_{1} $ and $ B_{2} $, we get the desired result.
   \end{proof}	
   \begin{theorem}\label{some2}
   	Assume that $ \varphi: U\subset \left[\left.0,\infty \right)  \right.\longrightarrow \mathbb{R}^{\alpha}  $ such that $ \varphi\in D_{\alpha} $ on $ Int(U) $ ($ Int(U) $ is the interior of $ U $) and $ \varphi^{(2\alpha)}\in C_{\alpha}[y_{1},y_{2}] $, where $ y_{1},y_{2}\in U $ with $ y_{1}<y_{1} $. If $ |\varphi| $ is generalized $ s $-convex on $ [y_{1},y_{2}] $, for some fixed $ 0<s\leq 1 $, then the following inequality holds:
   	\begin{eqnarray}
   	\left|  \frac{\Gamma(1+2\alpha)}{2^{\alpha}}\varphi\left( \frac{y_{1}+y_{2}}{2}\right) -\frac{\Gamma(1+2\alpha)[\Gamma(1+\alpha)]^{2}}{2^{\alpha}(y_{2}-y_{1})^{\alpha}}\,\ _{y_{1}}I^{(\alpha)}_{y_{2}}\varphi(x)\right| \nonumber\\&&\hspace{-4.2in}\leq \frac{(y_{2}-y_{1})^{2\alpha}}{16^{\alpha}}\left\lbrace \frac{2^{\alpha}\Gamma(1+(s+2)\alpha)}{\Gamma(1+(s+3)\alpha)}\left|  \varphi^{(2\alpha)}\left( \frac{y_{1}+y_{2}}{2}\right) \right| +\left[\frac{\Gamma(1+s\alpha)}{\Gamma(1+(s+1)\alpha)} \right. \right.\nonumber\\&&\hspace{-4.2in}\left. \left. -2^{\alpha}  \frac{\Gamma(1+(s+1)\alpha)}{\Gamma(1+(s+2)\alpha)}+\frac{\Gamma(1+(s+2)\alpha)}{\Gamma(1+(s+3)\alpha)}\right]  \left[\left|  \varphi^{(2\alpha)}(y_{1})\right|+\left|  \varphi^{(2\alpha)}(y_{2})\right|  \right] \right\rbrace\label{some3} \\&&\hspace{-4.2in} \leq\frac{(y_{2}-y_{1})^{2\alpha}}{16^{\alpha}}\left\lbrace\frac{2^{\alpha(2-s)}\Gamma(1+(s+2)\alpha)}{\Gamma(1+(s+3)\alpha)}\frac{\Gamma(1+s\alpha)\Gamma(1+\alpha)}{\Gamma(1+(s+1)\alpha)}+\frac{\Gamma(1+s\alpha)}{\Gamma(1+(s+1)\alpha)} \right. \nonumber\\&&\hspace{-4.2in}\left.-\frac{2^{\alpha}\Gamma(1+(s+1)\alpha)}{\Gamma(1+(s+2)\alpha)}+\frac{\Gamma(1+(s+2)\alpha)}{\Gamma(1+(s+3)\alpha)} \right\rbrace\left[\left| \varphi^{(2\alpha)}(y_{1})\right| +\left| \varphi^{(2\alpha)}(y_{2})\right|  \right]\label{some4}.
   	\end{eqnarray}
   \end{theorem}
   \begin{proof}
   	From Lemma \ref{some1}, we have
   	\begin{eqnarray*}
   		\left| \frac{\Gamma(1+2\alpha)}{2^{\alpha}}\varphi\left( \frac{y_{1}+y_{2}}{2}\right) -\frac{\Gamma(1+2\alpha)[\Gamma(1+\alpha)]^{2}}{2^{\alpha}(y_{2}-y_{1})^{\alpha}}\,\ _{y_{1}}I^{(\alpha)}_{y_{2}}\varphi(x)\right|
   		 \nonumber\\&&\hspace{-4.2in}\leq \frac{(y_{2}-y_{1})^{2\alpha}}{16^{\alpha}}\left[\,\ _{0}I^{(\alpha)}_{1}\gamma^{2\alpha}\left| \varphi^{(2\alpha)}(\gamma\frac{y_{1}+y_{2}}{2}+(1-\gamma)y_{1})\right| \right.\nonumber\\&&\hspace{-4.2in}\left.  +\,\ _{0}I^{(\alpha)}_{1}(\gamma-1)^{2\alpha}\left| \varphi^{(2\alpha)}\left( \gamma y_{2}+(1-\gamma)\frac{y_{1}+y_{2}}{2}\right)\right|  \right]
   		  \nonumber\\&&\hspace{-4.2in}\leq\frac{(y_{2}-y_{1})^{2\alpha}}{16^{\alpha}}\,\ _{0}I^{(\alpha)}_{1}\gamma^{2\alpha}\left[\gamma^{\alpha s}\left| \varphi^{(2\alpha)}\left( \frac{y_{1}+y_{2}}{2}\right)\right|  +(1-\gamma)^{\alpha s}\left| \varphi^{(2\alpha)}(y_{1})\right|  \right] \nonumber\\&&\hspace{-4.2in}+\frac{(y_{2}-y_{1})^{2\alpha}}{16^{\alpha}}\,\ _{0}I^{(\alpha)}_{1}(\gamma-1)^{2\alpha}\left[\gamma^{\alpha s}\left| \varphi^{(2\alpha)}(y_{2})\right|+(1-\gamma)^{\alpha s}\left| \varphi^{(2\alpha)}\left( \frac{y_{1}+y_{2}}{2}\right) \right| \right]\nonumber\\&&\hspace{-4.2in}= \frac{(y_{2}-y_{1})^{2\alpha}}{16^{\alpha}}\left\lbrace\frac{\Gamma(1+(s+2)\alpha)}{\Gamma(1+(s+3)\alpha)}\left| \varphi^{(2\alpha)}\left( \frac{y_{1}+y_{2}}{2}\right)\right| \right. \nonumber\\&&\hspace{-4.2in}\left.+\left[\frac{\Gamma(1+\alpha s)}{\Gamma(1+(s+1)\alpha)}-2^{\alpha}\frac{\Gamma(1+(s+1)\alpha)}{\Gamma(1+(s+2)\alpha)}+\frac{\Gamma(1+(s+2)\alpha)}{\Gamma(1+(s+3)\alpha)} \right] \left|\varphi^{(2\alpha)}(y_{1})\right|  \right\rbrace  \nonumber\\&&\hspace{-4.2in}+ \frac{(y_{2}-y_{1})^{2\alpha}}{16^{\alpha}}\left\lbrace\frac{\Gamma(1+(s+2)\alpha)}{\Gamma(1+(s+3)\alpha)}\left| \varphi^{(2\alpha)}\left( \frac{y_{1}+y_{2}}{2}\right) \right|  \right. \nonumber\\&&\hspace{-4.2in}+\left[\frac{\Gamma(1+\alpha s)}{\Gamma(1+(s+1)\alpha)}-2^{\alpha}\frac{\Gamma(1+(s+1)\alpha)}{\Gamma(1+(s+2)\alpha)}+\frac{\Gamma(1+(s+2)\alpha)}{\Gamma(1+(s+3)\alpha)} \right] \left| \varphi^{(2\alpha)}(a_{2})\right| \nonumber\\&&\hspace{-4.2in}+\left.\frac{\Gamma(1+(s+2)\alpha)}{\Gamma(1+(s+3)\alpha)}\left|\varphi^{(2\alpha)}\left( \frac{y_{1}+y_{2}}{2}\right) \right|\right\rbrace\nonumber\\&&\hspace{-4.2in}= \frac{(y_{2}-y_{1})^{2\alpha}}{16^{\alpha}}\left\lbrace\frac{2^{\alpha}\Gamma(1+(s+2)\alpha)}{\Gamma(1+(s+3)\alpha)}\left| \varphi^{(2\alpha)}\left( \frac{y_{1}+y_{2}}{2}\right)\right|  +\left[\frac{\Gamma(1+s\alpha)}{\Gamma(1+(s+1)\alpha)} \right.  \right. \nonumber\\&&\hspace{-4.2in}\left. \left.-2^{\alpha}\frac{\Gamma(1+(s+1)\alpha)}{\Gamma(1+(s+2)\alpha)}+\frac{\Gamma(1+(s+1)\alpha)}{\Gamma(1+(s+3)\alpha)} \right]\left[\left| \varphi^{(2\alpha)}(y_{1})\right|+\left| \varphi^{(2\alpha)}(y_{2})\right| \right]\right\rbrace .
   	\end{eqnarray*}
   	This proves inequality (\ref{some3}).
   	Since
   	\begin{equation*}
   	2^{\alpha(s-1)}\varphi^{(2\alpha)}\left( \frac{y_{1}+y_{2}}{2}\right) \leq\frac{\Gamma(1+s\alpha)\Gamma(1+\alpha)}{\Gamma(1+(s+1)\alpha)}\left( \varphi^{(2\alpha)}(y_{1})+\varphi^{(2\alpha)}(y_{2})\right) ,
   	\end{equation*}
   	then
   	\begin{eqnarray*}
   		\left| \frac{\Gamma(1+2\alpha)}{2^{\alpha}}\varphi\left( \frac{y_{1}+y_{2}}{2}\right) -\frac{\Gamma(1+2\alpha)[\Gamma(1+\alpha)]^{2}}{2^{\alpha}(y_{2}-y_{1})^{\alpha}}\,\ _{y_{1}}I^{(\alpha)}_{y_{2}}\varphi(x)\right|
   		  \nonumber\\&&\hspace{-4.2in}\leq\frac{(y_{2}-y_{1})^{2\alpha}}{16^{\alpha}} \left\lbrace\frac{2^{\alpha}\Gamma(1+(s+2)\alpha)}{\Gamma(1+(s+3)\alpha)}\frac{2^{-\alpha(s-1)}\Gamma(1+s\alpha)\Gamma(1+\alpha)}{\Gamma(1+(s+1)\alpha)}\left[\left| \varphi^{(2\alpha)}(y_{1})\right| +\left| \varphi^{(2\alpha)}(y_{2})\right|  \right] \right. \nonumber\\&&\hspace{-4.2in}\left. +\left[ \frac{\Gamma(1+s\alpha)}{\Gamma(1+(s+1)\alpha)} -\frac{2^{\alpha}\Gamma(1+(s+1)\alpha)}{\Gamma(1+(s+2)\alpha)}+\frac{\Gamma(1+(s+2)\alpha)}{\Gamma(1+(s+3)\alpha)}\right] \left[\left| \varphi^{(2\alpha)}(y_{1})\right| +\left| \varphi^{(2\alpha)}(y_{2})\right| \right] \right\rbrace \nonumber\\&&\hspace{-4.2in}=\frac{(y_{2}-y_{1})^{2\alpha}}{16^{\alpha}}\left\lbrace\frac{2^{\alpha(2-s)}\Gamma(1+(s+2)\alpha)}{\Gamma(1+(s+3)\alpha)}\frac{\Gamma(1+s\alpha)\Gamma(1+\alpha)}{\Gamma(1+(s+1)\alpha)}  + \frac{\Gamma(1+s\alpha)}{\Gamma(1+(s+1)\alpha)}\right. \nonumber\\&&\hspace{-4.2in}\left. -\frac{2^{\alpha}\Gamma(1+(s+1)\alpha)}{\Gamma(1+(s+2)\alpha)}+\frac{\Gamma(1+(s+2)\alpha)}{\Gamma(1+(s+3)\alpha)} \left[\left| \varphi^{(2\alpha)}(y_{1})\right| +\left| \varphi^{(2\alpha)}(y_{2})\right| \right] \right\rbrace
   	\end{eqnarray*}
   	Thus, we get the inequality (\ref{some4})and the proof is complete.
   \end{proof}
   \begin{remark}
   	\begin{enumerate}
   		\item When $ \alpha=1 $, Theorem \ref{some2} reduce to Theorem 2 in\cite{ozdemir2013some}.
   		\item If $ s=1 $ in Theorem \ref{some2} , then
   		\begin{eqnarray}
   		\left| \frac{\Gamma(1+2\alpha)}{2^{\alpha}}\varphi\left( \frac{y_{1}+y_{2}}{2}\right) -\frac{\Gamma(1+2\alpha)[\Gamma(1+\alpha)]^{2}}{2^{\alpha}(y_{2}-y_{1})^{\alpha}}\,\ _{y_{1}}I^{(\alpha)}_{y_{2}}\varphi(x)\right| \nonumber\\&&\hspace{-3.5in}\leq \frac{(y_{2}-y_{1})^{2\alpha}}{16^{\alpha}}\left\lbrace \frac{2^{\alpha}\Gamma(1+3\alpha)}{\Gamma(1+4\alpha)}\left| \varphi^{(2\alpha)}\left( \frac{y_{1}+y_{2}}{2}\right) \right| +\left[\frac{\Gamma(1+\alpha)}{\Gamma(1+2\alpha)} \right. \right.\nonumber\\&&\hspace{-3.4in}\left. \left. -2^{\alpha}  \frac{\Gamma(1+2\alpha)}{\Gamma(1+3\alpha)}+\frac{\Gamma(1+3\alpha)}{\Gamma(1+4\alpha)}\right]  \left[ \left| \varphi^{(2\alpha)}(y_{1})\right| +\left| \varphi^{(2\alpha)}(y_{2})\right| \right] \right\rbrace\nonumber \\&&\hspace{-3.5in} \leq\frac{(y_{2}-y_{1})^{2\alpha}}{16^{\alpha}}\left\lbrace\frac{2^{\alpha}\Gamma(1+3\alpha)}{\Gamma(1+4\alpha)}\frac{\left[ \Gamma(1+\alpha)\right]^{2}}{\Gamma(1+2\alpha)}+\frac{\Gamma(1+\alpha)}{\Gamma(1+2\alpha)} \right. \nonumber\\&&\hspace{-3.4in}\left.-\frac{2^{\alpha}\Gamma(1+2\alpha)}{\Gamma(1+3\alpha)}+\frac{\Gamma(1+3\alpha)}{\Gamma(1+4\alpha)} \right\rbrace\left[\left| \varphi^{(2\alpha)}(y_{1})\right| +\left| \varphi^{(2\alpha)}(y_{2})\right|  \right]\label{some5}.
   		\end{eqnarray}
   		\item If $ s=1 $ and $ \alpha=1 $ in Theorem \ref{some2}, then
   		\begin{eqnarray*}
   			\left| \varphi\left( \frac{y_{1}+y_{2}}{2}\right) -\frac{1}{y_{2}-y_{1}}\int_{y_{1}}^{y_{2}}\varphi(x)dx\right|\nonumber\\&&\hspace{-2.5in} \leq \frac{(y_{2}-y_{1})^{2}}{192}\left\lbrace 6\left| \varphi''\left( \frac{y_{1}+y_{2}}{2}\right) \right| +\left| \varphi''(y_{1})\right| +\left| \varphi''(y_{2})\right|  \right\rbrace \nonumber\\&&\hspace{-2.5in}\leq \frac{(y_{2}-y_{1})^{2}}{48}\left\lbrace \left| \varphi''(y_{1})\right| +\left| \varphi''(y_{2})\right|  \right\rbrace
   		\end{eqnarray*}
   	\end{enumerate}
   \end{remark}
   We give a new upper bound of the left generalized Hadamard's inequality for generalized $ s $-convex functions in the following theorem:
   \begin{theorem}\label{some6}
   	Assume that $ \varphi: U\subset \left[\left.0,\infty \right)  \right.\longrightarrow \mathbb{R}^{\alpha}  $ such that $ \varphi\in D_{\alpha} $ on $ Int(U) $ and $ \varphi^{(2\alpha)}\in C_{\alpha}[y_{1},y_{2}] $, where $ y_{1},y_{2}\in U $ with $ y_{1}<y_{2} $. If $ |\varphi^{(2\alpha)}|^{p_{2}} $ is generalized $ s $-convex on $ [y_{1},y_{2}] $, for some fixed $ 0<s\leq 1 $ and $ p_{2}>1 $ with $ \frac{1}{p_{1}}+\frac{1}{p_{2}}=1 $, then the following inequality holds:
   	\begin{eqnarray}\label{some7}
   	\left| \frac{\Gamma(1+2\alpha)}{2^{\alpha}}\varphi\left( \frac{y_{1}+y_{2}}{2}\right) -\frac{\Gamma(1+2\alpha)[\Gamma(1+\alpha)]^{2}}{2^{\alpha}(y_{2}-y_{1})^{\alpha}}\,\ _{y_{1}}I^{(\alpha)}_{y_{2}}\varphi(x)\right| \nonumber\\&&\hspace{-3.5in}\leq \frac{(y_{2}-y_{1})^{2\alpha}}{16^{\alpha}}\left[ \frac{\Gamma(1+s\alpha)}{\Gamma(1+(s+1)\alpha)}\right]^{\frac{1}{p_{2}}}\left[ \frac{\Gamma(1+2p_{1}\alpha)}{\Gamma(1+(2p_{1}+1)\alpha)}\right]^{\frac{1}{p_{1}}}\nonumber\\&&\hspace{-3.5in}\times\left[\left(\left| \varphi^{(2\alpha)}\left( \frac{y_{1}+y_{2}}{2}\right) \right| ^{p_{2}} +\left| \varphi^{(2\alpha)}(y_{1})\right| ^{p_{2}}\right)^{\frac{1}{p_{2}}}\right. \nonumber\\&&\hspace{-3in}\left.  +\left(\left| \varphi^{(2\alpha)}\left( \frac{y_{1}+y_{2}}{2}\right) \right| ^{p_{2}} +\left| \varphi^{(2\alpha)}(y_{2})\right| ^{p_{2}}\right)^{\frac{1}{p_{2}}} \right].
   	\end{eqnarray}
   \end{theorem}
   \begin{proof}
   	Let $ p_{1}>1 $, then from Lemma \ref{some1} and using generalized H\"{o}lder's inequality \cite{Yang2012}, we obtain
   	\begin{eqnarray*}
   		\left| \frac{\Gamma(1+2\alpha)}{2^{\alpha}}\varphi\left( \frac{y_{1}+y_{2}}{2}\right) -\frac{\Gamma(1+2\alpha)[\Gamma(1+\alpha)]^{2}}{2^{\alpha}(y_{2}-y_{1})^{\alpha}}\,\ _{y_{1}}I^{(\alpha)}_{y_{2}}\varphi(x)\right| \nonumber\\&&\hspace{-4.2in}\leq\frac{(y_{2}-y_{1})^{2\alpha}}{16^{\alpha}}\left\lbrace\,\ _{0}I^{(\alpha)}_{1}\gamma^{2\alpha}\left| \varphi^{(2\alpha)}\left( \gamma\frac{y_{1}+y_{2}}{2}+(1-\gamma)y_{1}\right) \right| \right.\nonumber\\&&\hspace{-4.2in}\left. \,\ +_{0}I^{(\alpha)}_{1}(\gamma-1)^{2\alpha}\left| \varphi^{(2\alpha)}\left( \gamma y_{2}+(1-\gamma)\frac{y_{1}+y_{2}}{2}\right) \right| \right\rbrace \nonumber\\&&\hspace{-4.2in}\leq\frac{(y_{2}-y_{1})^{2\alpha}}{16^{\alpha}}\left(\,\ _{0}I^{(\alpha)}_{1}\gamma^{2p_{1}\alpha} \right)^{\frac{1}{p_{1}}}\nonumber\\&&\hspace{-4in}\times \left(\,\ _{0}I^{(\alpha)}_{1}\left| \varphi^{(2\alpha)}\left( \gamma\frac{y_{1}+y_{2}}{2}+(1-\gamma)y_{1}\right) \right| ^{p_{2}}\right)^{\frac{1}{p_{2}}} \nonumber\\&&\hspace{-4.2in}+\frac{(y_{2}-y_{1})^{2\alpha}}{16^{\alpha}}\left(\,\ _{0}I^{(\alpha)}_{1}(1-\gamma)^{2p_{1}\alpha} \right)^{\frac{1}{p_{1}}}\nonumber\\&&\hspace{-4in}\times \left(\,\ _{0}I^{(\alpha)}_{1}\left| \varphi^{(2\alpha)}\left( \gamma y_{2}+(1-\gamma)\frac{y_{1}+y_{2}}{2}\right) \right| ^{p_{2}}\right)^{\frac{1}{p_{2}}}.
   	\end{eqnarray*}
   	Since $ \left| \varphi^{(2\alpha)}\right| ^{p_{2}} $ is generalized $ s $-convex, then
   	\begin{eqnarray*}
   		\,\ _{0}I^{(\alpha)}_{1}\left| \varphi^{(2\alpha)}\left( \gamma\frac{y_{1}+y_{2}}{2}+(1-\gamma)y_{1}\right) \right| ^{p_{2}}\nonumber\\&&\hspace{-3.2in} \leq \frac{\Gamma(1+s\alpha)}{\Gamma(1+(s+1)\alpha)}\left| \varphi^{(2\alpha)}\left( \frac{y_{1}+y_{2}}{2}\right) \right| ^{p_{2}}+\frac{\Gamma(1+s\alpha)}{\Gamma(1+(s+1)\alpha)}\left| \varphi^{(2\alpha)}(y_{1})\right| ^{p_{2}},
   	\end{eqnarray*}
   	which means
   	\begin{eqnarray*}
   		\,\	_{0}I^{(\alpha)}_{1}\left| \varphi^{(2\alpha)}\left( \gamma y_{2}+(1-\gamma)\frac{y_{1}+y_{2}}{2}\right) \right| ^{p_{2}}\nonumber\\&&\hspace{-2.5in} \leq \frac{\Gamma(1+s\alpha)}{\Gamma(1+(s+1)\alpha)}\left| \varphi^{(2\alpha)}(y_{2})\right| ^{p_{2}}\nonumber\\&&\hspace{-2.2in}+\frac{\Gamma(1+s\alpha)}{\Gamma(1+(s+1)\alpha)}\left| \varphi^{(2\alpha)}\left( \frac{y_{1}+y_{2}}{2}\right) \right| ^{p_{2}}.
   	\end{eqnarray*}
   	Hence
   	\begin{eqnarray*}
   		\left| \frac{\Gamma(1+2\alpha)}{2^{\alpha}}\varphi\left( \frac{y_{1}+y_{2}}{2}\right) -\frac{\Gamma(1+2\alpha)[\Gamma(1+\alpha)]^{2}}{2^{\alpha}(y_{2}-y_{1})^{\alpha}}\,\ _{y_{1}}I^{(\alpha)}_{y_{2}}\varphi(x)\right| \nonumber\\&&\hspace{-3.5in}\leq\frac{(y_{2}-y_{1})^{2\alpha}}{16^{\alpha}}\left[\frac{\Gamma(1+s\alpha)}{\Gamma(1+(s+1)\alpha)} \right]^{\frac{1}{p_{2}}} \left[\frac{\Gamma(1+2p_{1}\alpha)}{\Gamma(1+(2p_{1}+1)\alpha)} \right]^{\frac{1}{p_{1}}} \nonumber\\&&\hspace{-3.5in}\times\left\lbrace \left[\left| \varphi^{(2\alpha)}\left( \frac{y_{1}+y_{2}}{2}\right) \right| ^{p_{2}}+\left| \varphi^{(2\alpha)}(y_{1})\right| ^{p_{2}} \right]^{\frac{1}{p_{2}}}\right. \nonumber\\&&\hspace{-3in}\left. + \left[\left| \varphi^{(2\alpha)}\left( \frac{y_{1}+y_{2}}{2}\right) \right| ^{p_{2}}+\left| \varphi^{(2\alpha)}(y_{2})\right| ^{p_{2}} \right]^{\frac{1}{p_{2}}} \right\rbrace .
   	\end{eqnarray*}
   	The proof is complete.
   \end{proof}
   \begin{remark}
   	If $ s=1 $ in Theorem \ref{some6}, then
   	\begin{eqnarray}\label{some8}
   	\left| \frac{\Gamma(1+2\alpha)}{2^{\alpha}}\varphi\left( \frac{y_{1}+y_{2}}{2}\right) -\frac{\Gamma(1+2\alpha)[\Gamma(1+\alpha)]^{2}}{2^{\alpha}(y_{2}-y_{1})^{\alpha}}\,\ _{y_{1}}I^{(\alpha)}_{y_{2}}\varphi(x)\right| \nonumber\\&&\hspace{-3.5in}\leq\frac{(y_{2}-y_{1})^{2\alpha}}{16^{\alpha}}\left[\frac{\Gamma(1+\alpha)}{\Gamma(1+2\alpha)} \right]^{\frac{1}{p_{2}}} \left[\frac{\Gamma(1+2p_{1}\alpha)}{\Gamma(1+(2p_{1}+1)\alpha)} \right]^{\frac{1}{p_{1}}} \nonumber\\&&\hspace{-3.5in}\times\left\lbrace \left[\left| \varphi^{(2\alpha)}\left( \frac{y_{1}+y_{2}}{2}\right) \right| ^{p_{2}}+\left| \varphi^{(2\alpha)}(y_{1})\right| ^{p_{2}} \right]^{\frac{1}{p_{2}}}\right. \nonumber\\&&\hspace{-3in}\left. + \left[\left| \varphi^{(2\alpha)}\left( \frac{y_{1}+y_{2}}{2}\right) \right| ^{p_{2}}+\left| \varphi^{(2\alpha)}(y_{2})\right| ^{p_{2}} \right]^{\frac{1}{p_{2}}} \right\rbrace.
   	\end{eqnarray}
   \end{remark}
   \begin{corollary}
   	Assume that $ \varphi: U\subset \left[\left.0,\infty \right)  \right.\longrightarrow \mathbb{R}^{\alpha}  $ such that $ \varphi\in D_{\alpha} $ on $ Int(U) $ and $ \varphi^{(2\alpha)}\in C_{\alpha}[y_{1},y_{2}] $, where $ y_{1},y_{2}\in U $ with $ y_{1}<y_{1} $. If $ |\varphi^{(2\alpha)}|^{p_{2}} $ is generalized $ s $-convex on $ [y_{1},y_{2}] $, for some fixed $ 0<s\leq 1 $ and $ p_{2}>1 $ with $ \frac{1}{p_{1}}+\frac{1}{p_{2}}=1 $, then the following inequality holds:
   	\begin{eqnarray*}
   		\left| \frac{\Gamma(1+2\alpha)}{2^{\alpha}}\varphi\left( \frac{y_{1}+y_{2}}{2}\right) -\frac{\Gamma(1+2\alpha)[\Gamma(1+\alpha)]^{2}}{2^{\alpha}(y_{2}-y_{1})^{\alpha}}\,\ _{y_{1}}I^{(\alpha)}_{y_{2}}\varphi(x)\right| \nonumber\\&&\hspace{-4.2in}\leq\frac{(y_{2}-y_{1})^{2\alpha}}{16^{\alpha}}\frac{\left[ \Gamma(1+s\alpha)\right]^{\frac{1}{p_{2}}} }{\left[ \Gamma(1+(s+1)\alpha)\right]^{\frac{2}{p_{2}}} }\left[ \frac{\Gamma(1+2p_{1}\alpha)}{\Gamma(1+(2p_{1}+1)\alpha)}\right]^{\frac{1}{p_{1}}}\nonumber\\&&\hspace{-4.2in}\times\left\lbrace \left[\left(2^{\alpha(1-s)}\Gamma(1+s\alpha)\Gamma(1+\alpha) +\Gamma(1+(s+1)\alpha) \right)^{\frac{1}{p_{2}}}   \right. \right. \nonumber\\&&\hspace{-4.2in}\left. \left. +2^{\frac{\alpha(1-s)}{p_{2}}}\left[ \Gamma(1+\alpha)\right]^{\frac{1}{p_{2}}}\left[ \Gamma(1+\alpha)\right]^{\frac{1}{p_{2}}} \right] \left[\left| \varphi^{(2\alpha)}(y_{1})\right| +\left| \varphi^{(2\alpha)}(y_{2})\right|  \right] \right\rbrace
   	\end{eqnarray*}
   \end{corollary}
   \begin{proof}
   	Since $ \left| \varphi^{(2\alpha)}\right| ^{p_{2}} $ is generalized $ s $-convex, then
   	\begin{eqnarray*}
   		2^{\alpha(s-1)}\varphi^{(2\alpha)}\left( \frac{y_{1}+y_{2}}{2}\right) &\leq&\frac{\Gamma(1+s\alpha)\Gamma(1+\alpha)}{\Gamma(1+(s+1)\alpha)}\left(\varphi^{(2\alpha)}(y_{1})+\varphi^{(2\alpha)}(y_{2}) \right) .
   	\end{eqnarray*}
   	Hence using (\ref{some7}), we get
   	\begin{eqnarray*}
   		\left| \frac{\Gamma(1+2\alpha)}{2^{\alpha}}\varphi\left( \frac{y_{1}+y_{2}}{2}\right) -\frac{\Gamma(1+2\alpha)[\Gamma(1+\alpha)]^{2}}{2^{\alpha}(y_{2}-y_{1})^{\alpha}}\,\ _{y_{1}}I^{(\alpha)}_{y_{2}}\varphi(x)\right|
   		  \nonumber\\&&\hspace{-4.2in}\leq \frac{(y_{2}-y_{1})^{2\alpha}}{16^{\alpha}}\frac{\left[\Gamma(1+s\alpha) \right]^{\frac{1}{p_{2}}} }{\left[\Gamma(1+(s+1)\alpha) \right]^{\frac{2}{p_{2}}} }\left[\frac{\Gamma(1+2p_{1}\alpha)}{\Gamma(1+(2p_{1}+1)\alpha)} \right]^{\frac{1}{p_{1}}} \nonumber\\&&\hspace{-4.2in}\times\left\lbrace  \left[\left(2^{\alpha(1-s)}\Gamma(1+s\alpha)\Gamma(1+\alpha)+\Gamma(1+(s+1)\alpha) \right)|\varphi^{(2\alpha)}(y_{1})|^{p_{2}}\right. \right. \nonumber\\&&\hspace{-4.2in}\left. \left.  +2^{\alpha(1-s)}\Gamma(1+s\alpha)\Gamma(1+\alpha) |\varphi^{(2\alpha)}(y_{2})|^{p_{2}}\right]^{\frac{1}{p_{2}}} \right.\nonumber\\&&\hspace{-4.2in}\left.+\left[2^{\alpha(1-s)}\Gamma(1+s\alpha)\Gamma(1+\alpha) |\varphi^{(2\alpha)}(y_{1})|^{p_{2}}\right. \right. \nonumber\\&&\hspace{-4.2in}\left. \left. +\left(2^{\alpha(1-s)}\Gamma(1+s\alpha)\Gamma(1+\alpha)+\Gamma(1+(s+1)\alpha) \right)|\varphi^{(2\alpha)}(y_{2})|^{p_{2}} \right]^{\frac{1}{q}}  \right\rbrace
   	\end{eqnarray*}
   	and since $\displaystyle \sum_{i=1}^{k}(x_{i}+z_{i})^{\alpha n}\leq \displaystyle\sum_{i=1}^{k}x_{i}^{\alpha n}+\displaystyle\sum_{i=1}^{k}z_{i}^{\alpha n} $\\ for $ 0<n<1,x_{i},z_{i}\geq 0; \forall 1\leq i\leq k $, then we have
   	\begin{eqnarray*}
   		\left| \frac{\Gamma(1+2\alpha)}{2^{\alpha}}\varphi\left( \frac{y_{1}+y_{2}}{2}\right) -\frac{\Gamma(1+2\alpha)[\Gamma(1+\alpha)]^{2}}{2^{\alpha}(y_{2}-y_{1})^{\alpha}}\,\ _{y_{1}}I^{(\alpha)}_{y_{2}}\varphi(x)\right| \nonumber\\&&\hspace{-4.2in}\leq\frac{(y_{2}-y_{1})^{2\alpha}}{16^{\alpha}}\frac{\left[\Gamma(1+s\alpha) \right]^{\frac{1}{p_{2}}} }{\left[\Gamma(1+(s+1)\alpha) \right]^{\frac{2}{p_{2}}} }\left[\frac{\Gamma(1+2p_{1}\alpha)}{\Gamma(1+(2p_{1}+1)\alpha)} \right]^{\frac{1}{p_{1}}} \nonumber\\&&\hspace{-4.2in}\times\left\lbrace  \left[\left(2^{\alpha(1-s)}\Gamma(1+s\alpha)\Gamma(1+\alpha)+\Gamma(1+(s+1)\alpha) \right)^{\frac{1}{p_{2}}}\left| \varphi^{(2\alpha)}(y_{1})\right|\right. \right. \nonumber\\&&\hspace{-4.2in}\left. \left.   +2^{\frac{\alpha(1-s)}{p_{2}}}\left[ \Gamma(1+s\alpha)\right]^{\frac{1}{p_{2}}} \left[ \Gamma(1+\alpha)\right] ^{\frac{1}{p_{2}}} \left| \varphi^{(2\alpha)}(y_{2})\right| \right] \right.\nonumber\\&&\hspace{-4.2in}\left.+\left[2^{\frac{\alpha(1-s)}{p_{2}}}\left[ \Gamma(1+s\alpha)\right] ^{\frac{1}{p_{2}}}\left[ \Gamma(1+\alpha)\right] ^{\frac{1}{p_{2}}} \left| \varphi^{(2\alpha)}(y_{1})\right|\right. \right. \nonumber\\&&\hspace{-4.2in}\left. \left.  +\left(2^{\alpha(1-s)}\Gamma(1+s\alpha)\Gamma(1+\alpha)+\Gamma(1+(s+1)\alpha) \right)^{\frac{1}{p_{2}}}\left| \varphi^{(2\alpha)}(y_{2})\right|  \right] \right\rbrace
   	\end{eqnarray*}
   	where $ 0<\frac{1}{p_{2}}<1 $ for $ p_{2}>1 $.
   	By a simple calculation, we obtain the required result.
   \end{proof}
   Now, the generalized Hadamard's type inequality for generalized $ s $-concave functions.
   \begin{theorem}\label{some9}
   	Assume that $ \varphi: U\subset \left[\left.0,\infty \right)  \right.\longrightarrow \mathbb{R}^{\alpha}  $ such that $ \varphi\in D_{\alpha} $ on $ Int(U) $ and $ \varphi^{(2\alpha)}\in C_{\alpha}[y_{1},y_{2}] $, where $ y_{1},y_{2}\in U $ with $ y_{1}<y_{1} $. If $ |\varphi^{(2\alpha)}|^{p_{2}} $ is generalized $ s $-convex on $ [y_{1},y_{2}] $, for some fixed $ 0<s\leq 1 $ and $ p_{2}>1 $ with $ \frac{1}{p_{1}}+\frac{1}{p_{2}}=1 $, then the following inequality holds:
   	\begin{eqnarray*}
   		\left| \frac{\Gamma(1+2\alpha)}{2^{\alpha}}\varphi\left( \frac{y_{1}+y_{2}}{2}\right) -\frac{\Gamma(1+2\alpha)\left[\Gamma(1+\alpha) \right]^{2} }{2^{\alpha}(y_{2}-y_{1})^{\alpha}}\,\ _{y_{1}}I^{(\alpha)}_{y_{2}}\varphi(x)\right| \nonumber\\&&\hspace{-3.6in}\leq \frac{2^{\frac{\alpha(s-1)}{p_{2}}}(y_{2}-y_{1})^{2\alpha}}{16^{\alpha}\left( \Gamma(1+\alpha)\right) ^{\frac{1}{p_{2}}}}\left[\frac{\Gamma(1+2p_{1}\alpha)}{\Gamma(1+(2p_{1}+1)\alpha)} \right]^{\frac{1}{p_{1}}}\nonumber\\&&\hspace{-3.6in}\times \left[\left| \varphi^{(2\alpha)}\left( \frac{3y_{1}+y_{2}}{4}\right) \right| +\left| \varphi^{(2\alpha)}\left( \frac{y_{1}+3y_{2}}{4}\right) \right|  \right]
   	\end{eqnarray*}
   \end{theorem}
\newpage
   \begin{proof}
   	From Lemma \ref{some1} and using the generalized H\"{o}lder inequality for $ p_{2}>1 $ and $ \frac{1}{p_{1}}+\frac{1}{p_{2}}=1 $, we get
   	\begin{eqnarray*}
   		\left| \frac{\Gamma(1+2\alpha)}{2^{\alpha}}\varphi\left( \frac{y_{1}+y_{2}}{2}\right) -\frac{\Gamma(1+2\alpha)\left[\Gamma(1+\alpha) \right]^{2} }{2^{\alpha}(a_{2}-a_{1})^{\alpha}}\,\ _{y_{1}}I^{(\alpha)}_{y_{2}}\varphi(x)\right| \nonumber\\&&\hspace{-4.2in}\leq \frac{(y_{2}-y_{1})^{2\alpha}}{16^{\alpha}}\left[\,\ _{0}I^{(\alpha)}_{1}\gamma^{2\alpha}\left| \varphi^{(2\alpha)}\left( \gamma\frac{y_{1}+y_{2}}{2}+(1-\gamma)y_{1}\right) \right|  \right. \nonumber\\&&\hspace{-4.2in}+\left.\,\ _{0}I^{(\alpha)}_{1}(\gamma-1)^{2\alpha}\left| \varphi^{(2\alpha)}\left( \gamma y_{2}+(1-\gamma)\frac{y_{1}+y_{2}}{2}\right) \right|  \right] \nonumber\\&&\hspace{-4.2in}\leq \frac{(y_{2}-y_{1})^{2\alpha}}{16^{\alpha}}\left(\,\ _{0}I^{(\alpha)}_{1}\gamma^{2p_{1}\alpha} \right)^{\frac{1}{p_{1}}}\left(\,\ _{0}I^{(\alpha)}_{1}\left| \varphi^{(2\alpha)}\left( \gamma\frac{y_{1}+y_{2}}{2}+(1-\gamma)y_{1}\right) \right| ^{p_{2}} \right)^{\frac{1}{p_{2}}} \nonumber\\&&\hspace{-4.2in}+ \frac{(y_{2}-y_{1})^{2\alpha}}{16^{\alpha}}\left(\,\ _{0}I^{(\alpha)}_{1}(\gamma-1)^{2p_{1}\alpha} \right)^{\frac{1}{p_{1}}}\left(\,\ _{0}I^{(\alpha)}_{1}\left| \varphi^{(2\alpha)}\left( \gamma y_{2}+(1-\gamma)\frac{y_{1}+y_{2}}{2}\right) \right| ^{p_{2}} \right)^{\frac{1}{p_{2}}}.
   	\end{eqnarray*}
   	Since $ \left| \varphi^{(2\alpha)}\right| ^{p_{2}} $ is generalized $ s $-concave, then
   	\begin{eqnarray}\label{some10}
   	\,\ _{0}I^{(\alpha)}_{1}\left| \varphi^{(2\alpha)}\left( \gamma\frac{y_{1}+y_{2}}{2}+(1-\gamma)y_{1}\right) \right| ^{p_{2}}\leq \frac{2^{\alpha(s-1)}}{\Gamma(1+\alpha)}\left| \varphi^{(2\alpha)}\left( \frac{3y_{1}+y_{2}}{4}\right) \right| ^{p_{2}}
   	\end{eqnarray}
   	also
   	\begin{eqnarray}\label{some11}
   	\,\ _{0}I^{(\alpha)}_{1}\left| \varphi^{(2\alpha)}\left( \gamma y_{2}+(1-\gamma)\frac{y_{1}+y_{2}}{2}\right) \right| ^{p_{2}}\leq \frac{2^{\alpha(s-1)}}{ \Gamma(1+\alpha)}\left| \varphi^{(2\alpha)}\left( \frac{y_{1}+3y_{2}}{4}\right) \right| ^{p_{2}}.
   	\end{eqnarray}
   	From (\ref{some10}) and (\ref{some11}), we observe that
   	\begin{eqnarray*}
   		\left| \frac{\Gamma(1+2\alpha)}{2^{\alpha}}\varphi\left( \frac{y_{1}+y_{2}}{2}\right) -\frac{\Gamma(1+2\alpha)\left[\Gamma(1+\alpha) \right]^{2} }{2^{\alpha}(y_{2}-y_{1})^{\alpha}}\,\ _{y_{1}}I^{(\alpha)}_{y_{2}}\varphi(x)\right|
   		\nonumber\\&&\hspace{-4.2in}\leq \frac{(y_{2}-y_{1})^{2\alpha}}{16^{\alpha}} \left[\frac{\Gamma(1+2p_{1}\alpha)}{\Gamma(1+(2p_{1}+1)\alpha)} \right]^{\frac{1}{p_{1}}} \frac{2^{\frac{\alpha(s-1)}{p_{2}}}}{\left( \Gamma(1+\alpha)\right) ^{\frac{1}{p_{2}}}}\left| \varphi^{(2\alpha)}\left( \frac{3y_{1}+y_{2}}{4}\right) \right| 	\nonumber\\&&\hspace{-4.2in}+ \frac{(y_{2}-y_{1})^{2\alpha}}{16^{\alpha}} \left[\frac{\Gamma(1+2p_{1}\alpha)}{\Gamma(1+(2p_{1}+1)\alpha)} \right]^{\frac{1}{p_{1}}} \frac{2^{\frac{\alpha(s-1)}{p_{2}}}}{\left( \Gamma(1+\alpha)\right) ^{\frac{1}{p_{2}}}}\left| \varphi^{(2\alpha)}\left( \frac{y_{1}+3y_{2}}{4}\right) \right|
   		\nonumber\\&&\hspace{-4.2in}= \frac{2^{\frac{\alpha(s-1)}{p_{2}}}(y_{2}-y_{1})^{2\alpha}}{16^{\alpha}\left( \Gamma(1+\alpha)\right) ^{\frac{1}{p_{2}}}}\left[\frac{\Gamma(1+2p_{1}\alpha)}{\Gamma(1+(2p_{1}+1)\alpha)} \right]^{\frac{1}{p_{1}}} \left[\left| \varphi^{(2\alpha)}\left( \frac{3y_{1}+y_{2}}{4}\right) \right| +\left| \varphi^{(2\alpha)}\left( \frac{y_{1}+3y_{2}}{4}\right) \right|  \right]
   	\end{eqnarray*}
   	the proof is complete.
   \end{proof}
   \begin{remark}
   	\begin{enumerate}
   		\item If $ \alpha=1 $ in Theorem \ref{some9}, then
   		\begin{eqnarray*}
   			\left| \varphi\left( \frac{y_{1}+y_{2}}{2}\right) -\frac{1}{y_{2}-y_{1}}\int_{y_{1}}^{y_{2}}\varphi(x)dx\right| \nonumber\\&&\hspace{-3.2in} \leq\frac{2^{\frac{s-1}{q}}(y_{2}-y_{1})^{2}}{16}\left[\frac{1}{\Gamma(2p_{1}+1)} \right]^{\frac{1}{p_{1}}} \left[\left| \varphi''\left( \frac{3y_{1}+y_{2}}{4}\right) \right| +\left| \varphi''\left( \frac{y_{1}+3y_{2}}{4}\right) \right|  \right]
   		\end{eqnarray*}
   		\item If $ s=1 $ and $ \frac{1}{3}<\left[\frac{\Gamma(1+2p_{1}\alpha)}{\Gamma(1+(2p_{1}+1)\alpha)} \right]^{\frac{1}{p_{1}}} <1, p_{1}>1 $ in Theorem \ref{some9}, then
   		\begin{eqnarray*}
   			\left| \frac{\Gamma(1+2\alpha)}{2^{\alpha}}\varphi\left( \frac{y_{1}+y_{2}}{2}\right) -\frac{\Gamma(1+2\alpha)\left[\Gamma(1+\alpha) \right]^{2} }{2^{\alpha}(y_{2}-y_{1})^{\alpha}}\,\ _{y_{1}}I^{(\alpha)}_{y_{2}}\varphi(x)\right| \nonumber\\&&\hspace{-4in}\leq \frac{(y_{2}-y_{1})^{2\alpha}}{16^{\alpha}\left( \Gamma(1+\alpha)\right) ^{\frac{1}{p_{2}}}} \left[\left| \varphi^{(2\alpha)}\left( \frac{3y_{1}+y_{2}}{4}\right) \right| +\left| \varphi^{(2\alpha)}\left( \frac{y_{1}+3y_{2}}{4}\right) \right|  \right]
   		\end{eqnarray*}
   	\end{enumerate}
   \end{remark}
   \section{Applications to special means}
   As in \cite{Pecaric2000inequalities}, some generalized means are considered such as :\\
   $ A(y_{1},y_{2})=\frac{y_{1}^{\alpha}+y_{2}^{\alpha}}{2^{\alpha}} ,y_{1},y_{2}\geq 0;$\\
   $ L_{n}(y_{1},y_{2})=\left[\frac{\Gamma(1+n\alpha)}{\Gamma(1+(n+1)\alpha)}\left(y_{2}^{(n+1)\alpha} -y_{1}^{(n+1)\alpha}\right)  \right]^{\frac{1}{n}} , n\in \mathbb{Z}\{-1,0\},y_{1},y_{2}\in\mathbb{R}, y_{1}\neq y_{2}  $.\\

   In \cite{HuixiaSui2014}, the following example was given:\\
   let $ 0<s<1 $ and $ y_{1}^{\alpha},y_{2}^{\alpha},y_{3}^{\alpha}\in \mathbb{R}^{\alpha} $. Defining for $ x\in \mathbb{R}_{+} $,
   $$\varphi(b)=\begin{cases}
   y_{1}^{\alpha}& b=0,\\ y_{2}^{\alpha}b^{s\alpha}+y_{3}^{\alpha} & b>0.
   \end{cases}$$
   If $y_{2}^{\alpha}\geq 0^{\alpha} $ and $ 0^{\alpha}\leq y_{3}^{\alpha}\leq y_{1}^{\alpha} $, then $ \varphi\in GK^{2}_{s} $.\\
   \begin{proposition}
   	Let $ 0<y_{1}<y_{2} $ and $ s\in(0,1) $. Then
   	\begin{eqnarray*}
   		\left| \frac{\Gamma(1+2\alpha)}{2^{\alpha}}A^{s}(y_{1},y_{2})-\frac{\Gamma(1+2\alpha)\left[\Gamma(1+\alpha) \right]^{2} }{2^{\alpha}(y_{2}-y_{1})^{\alpha}}L^{s}_{s}(y_{1},y_{2})\right| \nonumber\\&&\hspace{-4in}\leq \frac{(y_{2}-y_{1})^{2\alpha}}{16^{\alpha}}\left| \frac{\Gamma(1+s\alpha)}{\Gamma(1+(s-2)\alpha)}\right| \left\lbrace\frac{2^{\alpha}\Gamma(1+3\alpha)\left[\Gamma(1+\alpha) \right]^{2} }{\Gamma(1+4\alpha)\Gamma(1+2\alpha)} \right. \nonumber\\&&\hspace{-3.8in}\left. + \frac{\Gamma(1+\alpha)}{\Gamma(1+2\alpha)}-\frac{2^{\alpha}\Gamma(1+2\alpha)}{\Gamma(1+3\alpha)}+\frac{\Gamma(1+3\alpha)}{\Gamma(1+4\alpha)}\right\rbrace \left[|y_{1}|^{(s-2)\alpha}+|y_{2}|^{(s-2)\alpha} \right]
   	\end{eqnarray*}
   \end{proposition}
   \begin{proof}
   	The result follows from Remark \ref{some5} (2) with $ \varphi:[0,1]\longrightarrow [0^{\alpha},1^{\alpha}] $, $ \varphi(x)=x^{s\alpha} $. and when $ \alpha=1 $, we have the following inequalitly:
   	\begin{eqnarray}\label{some12}
   	\left| A^{s}(y_{1},y_{2})-\frac{1}{y_{2}-y_{1}}L^{s}_{s}(y_{1},y_{2})\right| \leq \frac{(y_{2}-y_{1})^{2}|s(s-1)|}{48}\left\lbrace|y_{1}|^{s-2}+|y_{2}|^{s-2} \right\rbrace.
   	\end{eqnarray}
   \end{proof}
   \begin{proposition}
   	Let $ 0<y_{1}<y_{2} $ and $ s\in(0,1) $. Then
   	\begin{eqnarray*}
   		\left| \frac{\Gamma(1+2\alpha)}{2^{\alpha}}A^{s}(y_{1},y_{2})-\frac{\Gamma(1+2\alpha)\left[\Gamma(1+\alpha) \right]^{2} }{2^{\alpha}(y_{2}-y_{1})^{\alpha}}L^{s}_{s}(y_{1},y_{2})\right| \nonumber\\&&\hspace{-4in}\leq \frac{(y_{2}-y_{1})^{2\alpha}}{16^{\alpha}}\left[\frac{\Gamma(1+\alpha)}{\Gamma(1+2\alpha)} \right]^{\frac{1}{p_{2}}} \left| \frac{\Gamma(1+s\alpha)}{\Gamma(1+(s-2)\alpha)}\right| \left[ \frac{\Gamma(1+2p_{1}\alpha)}{\Gamma(1+(2p_{1}+1)\alpha)} \right] ^{\frac{1}{p_{1}}} \nonumber\\&&\hspace{-4in}\times\left[ \left(\left| \frac{y_{1}+y_{2}}{2}\right| ^{(s-2)p_{2}\alpha} +|y_{1}|^{(s-2)p_{2}\alpha}\right)^{\frac{1}{p_{2}}} +\left(\left| \frac{y_{1}+y_{2}}{2}\right| ^{(s-2)p_{2}\alpha} +|y_{2}|^{(s-2)p_{2}\alpha}\right)^{\frac{1}{p_{2}}} \right] ,
   	\end{eqnarray*}	
   	where $ p_{2}>1 $ and $ \frac{1}{p_{1}} +\frac{1}{p_{2}}=1$.
   \end{proposition}
   \begin{proof}
   	The result follows (\ref{some8}) with  $ \varphi:[0,1]\longrightarrow [0^{\alpha},1^{\alpha}] $, $ \varphi(x)=x^{s\alpha} $, and when $ \alpha=1 $, we have the following inequalitly:
   	\begin{eqnarray}\label{some13}
   	\left| A^{s}(y_{1},y_{2})-\frac{1}{y_{2}-y_{1}}L^{s}_{s}(y_{1},y_{2})\right| \nonumber\\&&\hspace{-2in}\leq \frac{(y_{2}-y_{1})^{2}|s(s-1)|}{2^{\frac{1}{p_{2}}}16(2p_{1}+1)^{\frac{1}{p_{1}}}}\left\lbrace\left( \left| \frac{y_{1}+y_{2}}{2}\right| ^{(s-2)p_{2}}+|y_{1}|^{(s-2)p_{2}}\right) ^{\frac{1}{p_{2}}}\right. \nonumber\\&&\hspace{-1in}\left. +\left( \left| \frac{y_{1}+y_{2}}{2}\right| ^{(s-2)p_{2}}+|y_{2}|^{(s-2)p_{2}}\right) ^{\frac{1}{p_{2}}} \right\rbrace.
   	\end{eqnarray}
   \end{proof}
   Where $ A(y_{1},y_{2}) $ and $ L_{n}(y_{1},y_{2}) $ in (\ref{some12}) and (\ref{some13}) are known as
   \begin{enumerate}
   	\item Arithmetic mean:
   	$$ A(y_{1},y_{2})=\dfrac{y_{1}+y_{2}}{2}, y_{1},y_{2}\in \mathbb{R}^{+} ;$$
   	\item Logarithmic mean :
   	$$ L(y_{1},y_{2})=\dfrac{y_{1}-y_{2}}{\ln|y_{1}|-\ln|y_{2}|},|y_{1}|\neq y_{2},y_{1},a_{2}\neq 0, y_{1},y_{2}\in \mathbb{R}^{+}; $$
   	Generalized Log-mean:
   	$$ L_{n}(y_{1},y_{2})=\left[ \dfrac{y_{2}^{n+1}-y_{1}^{n+1}}{(n+1)(y_{2}-y_{1})}\right]^{\frac{1}{n}} , n\in \mathbb{Z}\setminus \left\lbrace -1,0\right\rbrace ,y_{1},y_{2}\in \mathbb{R}^{+}. $$
   \end{enumerate}
   Now, we give application to wave equation on Cantor sets:\\
   the wave equation on Cantor sets (local fractional wave equation) was given by \cite{Yang2012}
   \begin{equation}\label{W1}
   \frac{\partial^{2\alpha}f(x,t)}{\partial t^{2\alpha}}=A^{2\alpha}\frac{\partial^{2\alpha}f(x,t)}{\partial x^{2\alpha}}
   \end{equation}
   Following (\ref{W1}), a wave equation on Cantor sets was proposed as follows \cite{Yang2013Local}:
   \begin{equation}\label{W2}
   \frac{\partial^{2\alpha}f(x,t)}{\partial t^{2\alpha}}=\frac{x^{2\alpha}}{\Gamma(1+2\alpha)}\frac{\partial^{2\alpha}f(x,t)}{\partial x^{2\alpha}},\ \ \ 0\leq\alpha\leq 1
   \end{equation}
   where $ f(x,t) $ is a fractal wave function and the initial value is given by  $  f(x,0)=\dfrac{x^{\alpha}}{\Gamma(1+\alpha)} $. The solution of (\ref{W2}) is given as  $$   f(x,t)=\frac{x^{\alpha}}{\Gamma(1+\alpha)}+\frac{t^{2\alpha}}{\Gamma(1+2\alpha)}. $$
   By using (\ref{some1}), we have
   \begin{eqnarray*}
   	 \frac{\Gamma(1+2\alpha)\Gamma(1+\alpha)}{2^{\alpha}(y_{2}-y_{1})^{\alpha}}\int_{y_{1}}^{y_{2}}f(x,t)(dt)^{\alpha}-\frac{\Gamma(1+2\alpha)}{2^{\alpha}}f\left( x,\frac{y_{1}+y_{2}}{2}\right)&&\nonumber\\ = \frac{(y_{2}-y_{1})^{\alpha}}{8^{\alpha}\Gamma(1+2\alpha)}\left[\left(\frac{2}{y_{2}-y_{1}} \right)^{2\alpha} \,\ _{y_{1}}I^{(\alpha)}_{\frac{y_{2}+y_{1}}{2}} (t-y_{1})^{2\alpha}x^{2\alpha}\frac{\partial^{2\alpha}f(x,t) }{\partial x^{\alpha}}\right. \nonumber\\ \left. +\,\ _{y_{1}}I^{(\alpha)}_{\frac{y_{2}+y_{1}}{2}} \left( \frac{2(t-y_{1})}{y_{2}-y_{1}}-1\right)^{2\alpha} x^{2\alpha}\frac{\partial ^{2\alpha}f(x,t)}{\partial x^{\alpha}} \right]
   \end{eqnarray*}

\end{document}